\documentclass[10pt]{amsart}
\pdfoutput=1
\usepackage[latin1]{inputenc}
\usepackage[T1]{fontenc}
\usepackage{indentfirst}
\usepackage{textcomp}
\usepackage[english]{babel}
\usepackage{amsmath,amsfonts,amssymb,amsopn,amscd,amsthm}
\usepackage{mathbbol}
\usepackage[enableskew]{youngtab}	
\usepackage{mathrsfs}
\usepackage{graphicx}
\usepackage[dvipsnames]{xcolor}
\usepackage{pstricks,pst-plot}
\usepackage{xy}  \xyoption{all}
\usepackage{pb-diagram,pb-xy}
\usepackage{enumerate}

\begin{document}

\author{Pierre-Lo\"ic M\'eliot}
\title[Kerov's central limit theorem for Schur-Weyl measures]{Kerov's central limit theorem for Schur-Weyl \\ measures of parameter $\alpha=1/2$}
\address{The Gaspard--Monge Institut of electronic and computer science,
University of Marne-La-Vall\'ee Paris-Est,
77454 Marne-la-Vall\'ee Cedex 2, France}
\email{meliot@phare.normalesup.org}

\newcommand{\Z}{\mathbb{Z}}     
\newcommand{\N}{\mathbb{N}}    
\newcommand{\R}{\mathbb{R}}    
\newcommand{\C}{\mathbb{C}}    
\newcommand{\I}{\mathrm{i}}
\newcommand{\sym}{\mathfrak{S}}
\newcommand{\Comp}{\mathfrak{C}}   
\newcommand{\Part}{\mathfrak{P}}      
\newcommand{\proba}{\mathbb{P}}     
\newcommand{\esper}{\mathbb{E}}     
\newcommand{\card}{\mathrm{card}}  
\newcommand{\obs}{\mathscr{O}}
\newcommand{\GL}{\mathrm{GL}}
\newcommand{\SWa}{\mathrm{SW}_{n,\alpha,c}}
\newcommand{\wt}{\mathrm{wt}}
\newcommand{\eps}{\varepsilon}
\newcommand{\tilh}{\widetilde{h}}
\newcommand{\tilp}{\widetilde{p}}
\newcommand{\tilq}{\widetilde{q}}
\newcommand{\tr}{\mathrm{tr}}
\newcommand{\lle}{\left[\!\left[}              
\newcommand{\rre}{\right]\!\right]}    
\newcommand{\scal}[2]{\left\langle #1\vphantom{#2}\,\right |\left.#2 \vphantom{#1}\right\rangle}   
\newtheorem{theorem}{Theorem}
\newtheorem{proposition}[theorem]{Proposition}
\newtheorem{lemma}[theorem]{Lemma}
\newtheorem{corollary}[theorem]{Corollary}
\theoremstyle{remark}
\newtheorem*{example}{Example}
\newtheorem*{examples}{Examples}
\newcommand{\figcap}[2]{\begin{figure}[ht] \begin{center} 
							{\footnotesize{#1}} 
							\caption{#2} \end{center} \end{figure}}
\newcommand{\comment}[1]{}

\begin{abstract}
In this article, we show that Kerov's central limit theorem related to the fluctuations of Young diagrams under the Plancherel measure extends to the case of Schur-Weyl measures, which are the probability measures on partitions associated to the representations of the symmetric groups $\sym_{n}$ on tensor products of vector spaces $V^{\otimes n}$ (see \cite{Bia01}). The proof is inspired by the one given in \cite{IO02} for the Plancherel measure, and it relies on the combinatorics of the algebra of observables of diagrams; we shall also use \'Sniady's theory of cumulants of observables (\emph{cf.} \cite{Sni06}).
\end{abstract}
\maketitle

Given a finite group $G$ and a finite-dimensional complex linear representation $V$ of $G$, the decomposition in irreducible components $V=\bigoplus_{\lambda \in \widehat{G}}m_{\lambda}\,V_{\lambda}$ yields a probability measure on the set $\widehat{G}$ of isomorphism classes of irreducible representations of $G$:
$$\proba_{V}\left[\lambda \in \widehat{G}\right]=\frac{m_{\lambda}\,\dim V_{\lambda}}{\dim V}$$
In particular, when $G=\sym_{n}$ is the symmetric group of permutations of $\lle 1,n\rre$, a reducible representation of $\sym_{n}$ yields a probability measure on the set $\Part_{n}=\widehat{\sym}_{n}$ of integer partitions of size $n$. When $V=\C\sym_{n}$ is the (left) regular representation, one obtains the so-called Plancherel measure, which has been extensively studied in the $n \to \infty$ asymptotic; see for instance \cite{IO02}. In this paper, we focus on the \textbf{Schur-Weyl measures} that correspond to the representations $(\C^{N})^{n} \curvearrowleft \sym_{n}$, where $\sym_{n}$ acts on the right by permutation of the letters of the simple tensors:
$$(v_{1}\otimes v_{2}\otimes \cdots \otimes v_{n})\cdot \sigma = v_{\sigma(1)}\otimes v_{\sigma(2)} \otimes \cdots \otimes v_{\sigma(n)}$$
By Schur-Weyl reciprocity, the probability measure on partitions associated to this representation is also a probability measure on irreducible representations of $\GL(N,\C)$. The first order asymptotic of these measures when $n^{1/2} \simeq cN$ and $n$ goes to infinity has been determined by P. Biane in \cite{Bia01}; we shall recall this in \S\ref{swbiane}, and by using the Robinson-Schensted-Knuth correspondence, we will also give a ``combinatorial model'' for these probability measures.\bigskip

The asymptotics of the Plancherel measures on partitions correspond to the case $c=0$, and for Plancherel measures, S. Kerov has shown that Young diagrams deviate from their limit shape by a (generalized) gaussian process (\cite{Ker93}, \cite{IO02}). Consequently, it was natural to ask whether such a result extends to the case of Schur-Weyl measures with parameter $c>0$. We have found out that it is indeed the case, and that the central limit theorem for Schur-Weyl measures is up to a translation on the $x$-axis the same as for Plancherel measures. This is the main result of this paper, see Theorem \ref{global} in \S\ref{ksw}. Notice that our result is the ``Schur-Weyl dual'' of a result of Lytova and Pastur related to the deviation of the empirical spectral measure of Wishart matrices from the Mar\v{c}enko-Pastur distribution, \emph{cf.} \cite{LP09}. In order to prove Theorem \ref{global}, we shall follow the approach of V. Ivanov and G. Olshanski (\cite{IO02}), and we shall therefore use the algebra $\obs$ of \textbf{observables of diagrams}. These objects are in some sense ``polynomial functions on partitions'' (see \cite{KO94}), and in the setting of random partitions, their expectations play the same role as the moments for real or complex random variables. Moreover, P. \'Sniady has developed in \cite{Sni06} a theory of \textbf{cumulants of observables} of diagrams that allow to prove results of gaussian concentration. These facts are recalled in \S\ref{obs} and \S\ref{sniady}, and we will also compute certain observables of the limit shapes $\Omega_{c}$ of the Young diagrams under Schur-Weyl measures with parameters $\alpha=1/2,c>0$, \emph{cf.} \S\ref{limitshape}.\bigskip

The core of the proof of our central limit theorem is contained in \S\ref{core}, and it is shown that Chebyshev polynomials of the second kind yield random variables related to the second order asymptotic of the Schur-Weyl measures, and that are asymptotically centered independant gaussian variables. Our proof is slightly different from the one given in \cite[\S7]{IO02} for Plancherel measures, because in the setting of Schur-Weyl measures, one cannot use the so-called Kerov degree on observables. Hence, one has to use instead the weight filtration on observables, and the arguments of Ivanov and Olshanski that relie on Kerov's degree will be replaced by a somewhat nasty trick, see Lemma \ref{nasty}. Thus, our calculations present quite miraculous simplifications, and these simplifications can be summarized by the following fact: Chebyshev polynomials of the second kind do play a hidden and nethertheless prominent role in the combinatorics of observables of diagrams. We ought to say that we do not understand completely the reasons of the universality of Kerov's central limit theorem; in this paper, we only show that calculations can be performed in a more general setting than previously known, but as for now we do not know why.\bigskip\bigskip

\section{Schur-Weyl measures of parameter $\alpha=1/2$}\label{swbiane}
In the following, we fix two positive real numbers $\alpha$ and $c$. Recall that the irreducible representations of the symmetric group $\sym_{n}$ are labelled by partitions of size $n$, that is to say, non-increasing sequences 
$$\lambda=(\lambda_{1},\ldots,\lambda_{r})$$
of positive integers that sum up to $n$; see for instance \cite[Chapter 4]{Weyl39}. Such a sequence can be represented by its \textbf{Young diagram}, which is an array of $n$ boxes with $\lambda_{1}$ boxes in the first line, $\lambda_{2}$ boxes in the second line, etc.
\figcap{\yng(1,2,2,5,7)}{The Young diagram of the partition $\lambda=(7,5,2,2,1)$.}\bigskip

 We denote by $\Part_{n}$ the set of integer partitions of size $n$, and if $\lambda=(\lambda_{1},\ldots,\lambda_{r})$ is in $\Part_{n}$, we set $|\lambda|=n$ and $\ell(\lambda)=r$. We also denote by $\dim \lambda=\dim V_{\lambda}$ the dimension of the irreducible representation of $\sym_{n}$ of label $\lambda$; this dimension can be computed by using the hook length formula, \emph{cf.} \cite[Chapter 1, \S7]{Mac95}. Alternatively, by using the branching rules for the irreducible representations of the symmetric groups $\sym_{n}$, it is easy to show that 
$$\dim \lambda=\card\big\{\text{standard tableaux of shape }\lambda\big\}.$$
Here, a \textbf{standard tableau} of shape $\lambda$ is a numbering of the cases of the Young diagram of $\lambda$ by the integers $1,2,\ldots,n$, such that each row and each column is strictly increasing.
\begin{example}
The five standard tableaux of shape $(3,2)$ are drawn in figure \ref{standardtableaux}. Hence, $\dim(3,2)=5$.
\figcap{\young(45,123)\qquad\qquad\young(35,124)\qquad\qquad\young(34,125)\qquad\qquad\young(25,134)\qquad\qquad\young(24,135)}{The standard tableaux of shape $(3,2)$.\label{standardtableaux}}
\end{example}

 The \textbf{Schur-Weyl measure} of parameters $\alpha$ and $c$ on partitions of size $n$ is the probability measure on $\Part_{n}$ defined by:
$$\SWa[\lambda]=\frac{m_{\lambda}\,\dim \lambda}{N^{n}}$$
where $n^{\alpha}\simeq cN$ and the multiplicities $m_{\lambda}$ are the one coming from the decomposition of the (right-side) representation $(\C^{N})^{\otimes n} \curvearrowleft \sym_{n}$. Of course, these probability measures depend on the exact sizes $N$, but for the asymptotic analysis, we shall see that the parameters $\alpha$ and $c$ are sufficient. \bigskip

The general linear group $\GL(N,\C)$ acts on the left of the tensor product $V=(\C^{N})^{\otimes n}$ by
$$g\cdot(v_{1}\otimes v_{2}\otimes \cdots \otimes v_{n})=g(v_{1}) \otimes g(v_{2}) \otimes \cdots \otimes g(v_{n}),$$
and it is well-known that the algebras generated by the actions of $\GL(N,\C)$ and $\sym_{n}$ in $\mathrm{End}(V)$ are full mutual centralizers --- this is the Schur-Weyl duality phenomenon, see \cite[Chapter 4]{Weyl39}. Hence, $V$ can be decomposed as a direct sum of irreducible $(\GL(N,\C),\sym_{n})$-bimodules:
$$_{\GL(N,\C)\curvearrowright\!}\left\{(\C^{N})^{\otimes n}\right\}_{\!\curvearrowleft \sym_{n}}=\bigoplus_{\lambda}{\,\,}_{\GL(N,\C)\curvearrowright\!\!}\left\{M_{\lambda}\right\}\otimes_{\C}\left\{V_{\lambda}\right\}_{\!\curvearrowleft \sym_{n}}$$
The irreducible algebraic representations of $\GL(N,\C)$ are labelled by the sequences $\lambda=(\lambda_{1},\ldots,\lambda_{N})$ of non-negative integers, and consequently, the Schur-Weyl measure charges only the partitions of length $\ell(\lambda)\leq N$; moreover, $m_{\lambda}=\dim M_{\lambda}$. These dimensions can be given a combinatorial interpretation similar to the one for the $\dim V_{\lambda}$'s, see \cite[Appendix A, \S8]{Mac95}. Hence,
$$\dim M_{\lambda}=\card\big\{\text{semistandard tableaux of shape }\lambda\text{ with entries in }\lle 1,N\rre\big\},$$
where a \textbf{semistandard tableau} of shape $\lambda$ is a numbering of the cases of the Young diagram of $\lambda$ that is non-decreasing along the rows and strictly increasing along the columns. 
\begin{example}
The two semistandard tableaux of shape $(3,2)$ and entries in $\lle 1,2\rre$ are drawn in figure \ref{semistandardtableaux}. As a consequence, the dimension of the $\GL(2,\C)$-module $M_{3,2}$ is $2$.
\figcap{\young(22,111)\qquad\qquad\qquad\young(22,112)}{The semistandard tableaux of shape $(3,2)$ and entries in $\lle 1,2\rre$.\label{semistandardtableaux}}
\end{example}

Given a word of length $n$ over the alphabet $\lle 1,N\rre$, one can associate to it two tableaux of size $n$ and same shape by using the \textbf{Robinson-Schensted-Knuth algorithm}, see \cite{Ful97}. For instance, when $n=9$ and $N=5$, the word $m=233154243$ is associated to the two tableaux\vspace{2mm}
$$ \young(5,234,12334)\qquad\text{and}\qquad \young(7,469,12358)\,.$$
\vspace{0.1mm}

\noindent The first tableau is semistandard with entries in $\lle 1,N\rre$, and can be obtained by applying the jeu de taquin algorithm to the ribbon of the word $m$; the second tableau is always a standard tableau. Hence, there is a bijection
$$\mathrm{RSK} : \big\{\text{words of length }n\text{ over }\lle 1,N\rre \big\}\to \bigsqcup_{\lambda \in \Part_{n}}\mathcal{ST}(\lambda) \times \mathcal{SST}(\lambda,N),$$
and if one reminds only the shape $\mathrm{RSKsh}(m)$ of the tableaux associated to a word $m$ by $\mathrm{RSK}$, then the Schur-Weyl measure is precisely the image of the uniform law on $N^{n}$ by the map $\mathrm{RSKsh}$. Hence, one obtains a combinatorial model for the measures $\SWa$, and moreover, it can be shown that the size of the first part of $\mathrm{RSKsh}(m)$ is always equal to the length of a non-decreasing subsequence of maximal length in $m$; consequently, the study of Schur-Weyl measures is related to analogues of Ulam's problem (\cite{Ulam61}).\bigskip

If $\lambda$ is a Young diagram, we turn it by 45 degrees and we consider the upper boundary of the new drawing as a continuous function $s\mapsto \lambda(s)$, with by convention $\lambda(s)=|s|$ if $s$ is too big (see figure \ref{russiandiagram}).
\figcap{\includegraphics{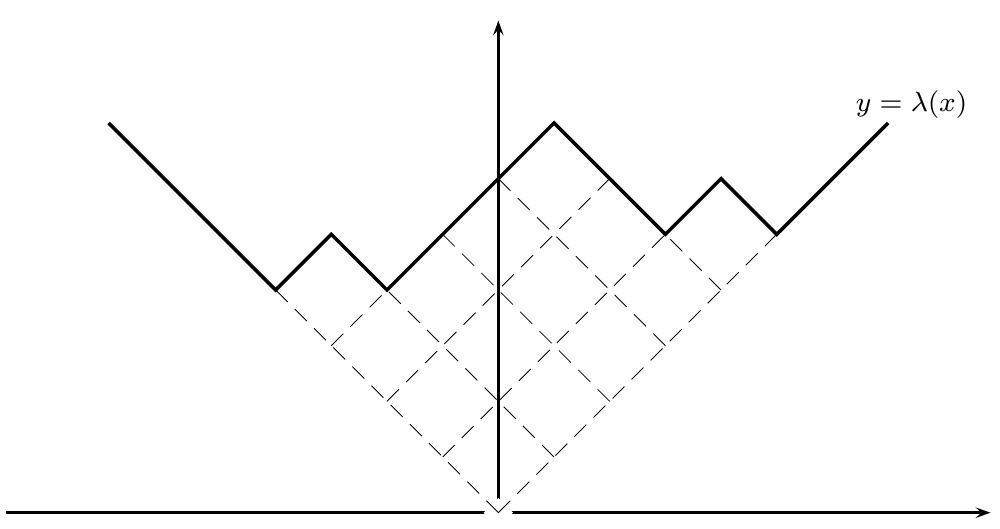}
}{Function $s\mapsto \lambda(s)$ associated to the Young diagram $\lambda=(5,4,4,1)$.\label{russiandiagram}}

\comment{\psset{unit=1mm}
\pspicture(-50,0)(50,50)
\psline{->}(-50,0)(50,0)
\psline{->}(0,0)(0,50)
\rput*[0,0]{45}(0,0){\psline[linestyle=dashed, linewidth=0.25pt](0,0)(0,32) \psline[linestyle=dashed, linewidth=0.25pt](8,0)(8,32) \psline[linestyle=dashed, linewidth=0.25pt](16,0)(16,24) \psline[linestyle=dashed, linewidth=0.25pt](24,0)(24,24) \psline[linestyle=dashed, linewidth=0.25pt](32,0)(32,24) \psline[linestyle=dashed, linewidth=0.25pt](40,0)(40,8)
\psline[linestyle=dashed, linewidth=0.25pt](0,0)(40,0) \psline[linestyle=dashed, linewidth=0.25pt](0,8)(40,8) \psline[linestyle=dashed, linewidth=0.25pt](0,16)(32,16) \psline[linestyle=dashed, linewidth=0.25pt](0,24)(32,24) \psline[linestyle=dashed, linewidth=0.25pt](0,32)(8,32)
\psline[linewidth=1pt](0,56)(0,32)(8,32)(8,24)(32,24)(32,8)(40,8)(40,0)(56,0)
}
\rput(42,41.5){$y=\lambda(x)$}
\endpspicture}

This interpretation permits to rescale the Young diagrams, and in particular, we shall denote by $\lambda^{*}$ the Young diagram $\lambda$ rescaled by a factor $1/\sqrt{|\lambda|}$ in both directions: 
$$\lambda^{*}(s)=\lambda(s\sqrt{n})/\sqrt{n}$$
Then, P. Biane has shown that under Schur-Weyl measures with parameter $\alpha \geq 1/2$, rescaled Young diagrams admit a limit form (see Theorem 3 in \cite{Bia01}):
\figcap{\includegraphics{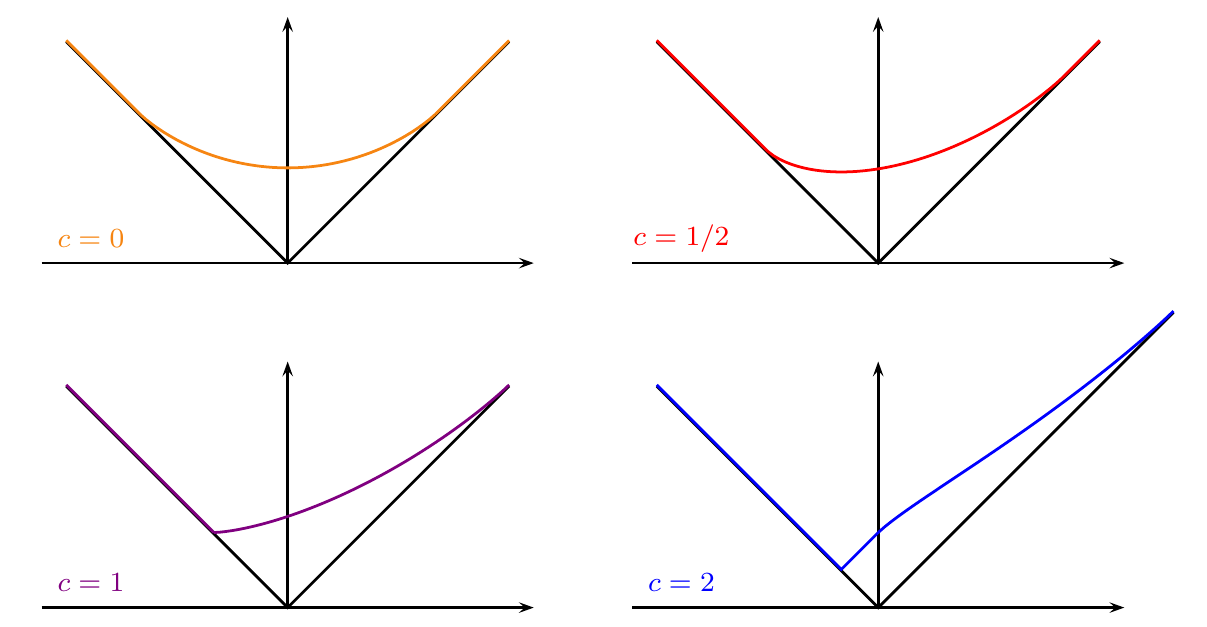}}{Limit shapes of the rescaled Young diagrams under Schur-Weyl measures of parameter $\alpha=1/2$.}

\comment{\psset{unit=0.5mm}
\pspicture(-50,-73)(170,53)
\psline{->}(-50,0)(50,0)
\psline{->}(-50,-70)(50,-70)
\psline{->}(70,0)(170,0)
\psline{->}(70,-70)(170,-70)
\psline{->}(0,0)(0,50)
\psline{->}(0,-70)(0,-20)
\psline{->}(120,0)(120,50)
\psline{->}(120,-70)(120,-20)
\psline(-45,45)(0,0)(45,45)
\psline(-45,-25)(0,-70)(45,-25)
\psline(75,45)(120,0)(165,45)
\psline(75,-25)(120,-70)(180,-10)
\rput(-40,5){\textcolor{BurntOrange}{$c=0$}}
\rput(-40,-65){\textcolor{violet}{$c=1$}}
\rput(80,5){\textcolor{red}{$c=1/2$}}
\rput(80,-65){\textcolor{blue}{$c=2$}}
\parametricplot[linecolor=BurntOrange]{-2}{2}{t 15 mul
t 2 div arcsin t mul 3.14159 mul 180 div t t mul -1 mul 4 add sqrt add 30 3.14159 div mul 0.25 add}
\psline[linecolor=BurntOrange](-45,45.25)(-30,30.25)
\psline[linecolor=BurntOrange](45,45.25)(30,30.25)
\parametricplot[linecolor=red]{-1.5}{2.5}{t 15 mul 120 add
t 0.5 add t 2 div 1 add sqrt 2 mul div arcsin t mul 1.75 t 2 div add t 2 div 1 add sqrt 2 mul div arccos 2 mul add 3.14159 mul 180 div  t 0.5 -1 mul add t 0.5 -1 mul add mul -1 mul 4 add sqrt 2 div add 30 3.14159 div mul 0.25 add}
\psline[linecolor=red](97.5,22.75)(75,45.25)
\psline[linecolor=red](165,45.25)(157.5,37.75)
\parametricplot[linecolor=violet]{-1}{3}{t 15 mul
t 1 add sqrt 2 div arcsin t mul t 1 add sqrt 2 div arccos add 3.14159 mul 180 div  t -1 add t -1 add mul -1 mul 4 add sqrt 2 div add 30 3.14159 div mul -69.75 add}
\psline[linecolor=violet](-15,-54.75)(-45,-24.75)
\parametricplot[linecolor=blue]{0}{4}{t 15 mul 120 add
t 2 add t 2 mul 1 add sqrt 2 mul div arcsin t mul t 2 mul -2 add t 2 mul 1 add sqrt 2 mul div arccos 0.5 mul add 3.14159 mul 180 div  t t mul -1 mul 4 t mul add sqrt 2 div add 30 3.14159 div mul -69.75 add}
\psline[linecolor=blue](75,-24.75)(112.5,-62.25)(120,-54.75)
\endpspicture
}

\begin{theorem}[First order asymptotics of Schur-Weyl measures of parameter $\alpha=1/2$]\label{firstasymptotic}
We fix $\eps>0$, and we suppose that $\alpha=1/2$. As $n$ goes to infinity,
$$\SWa[\|\lambda^{*}-\Omega_{c} \|_{\infty} \geq \eps]\to 0,$$ 
where $\Omega_{c}$ is defined by the following formulas:
\begin{align*}&\Omega_{0}(s)=\Omega(s)=\begin{cases}
            \frac{2}{\pi} \left(s \arcsin(\frac{s}{2})+\sqrt{4-s^2} \right) &\text{if }|s| \leq 2,\\
|s| &\text{otherwise;}
            \end{cases}\\
&\Omega_{c\in \,]0,1[}(s)=\begin{cases}
            \frac{2}{\pi} \left(s \arcsin(\frac{s+c}{2\sqrt{1+sc}})+\frac{1}{c}\arccos(\frac{2+sc-c^{2}}{2\sqrt{1+sc}})+\frac{\sqrt{4-(s-c)^2}}{2} \right) &\text{if }|s-c| \leq 2,\\
|s| &\text{otherwise;}
            \end{cases}\\
            &\Omega_{1}(s)=\begin{cases}
            \frac{s+1}{2}+\frac{1}{\pi}\left((s-1)\arcsin(\frac{s-1}{2})+\sqrt{4-(s-1)^{2}}\right)&\text{if }|s-1|\leq 2,\\
            |s|&\text{otherwise;}
            \end{cases}\\
            &\Omega_{c >1}(s)=\begin{cases}s+\frac{2}{c}&\text{if } s\in\,]\frac{-1}{c},c-2[\,,\\
            \frac{2}{\pi} \left(s \arcsin(\frac{s+c}{2\sqrt{1+sc}})+\frac{1}{c}\arccos(\frac{2+sc-c^{2}}{2\sqrt{1+sc}})+\frac{\sqrt{4-(s-c)^2}}{2} \right) &\text{if }|s-c| \leq 2,\\
|s| &\text{otherwise.}
            \end{cases}
\end{align*}
The case when $c=0$ corresponds to the limit shape of rescaled Young diagrams under Schur-Weyl measures with parameter $\alpha>1/2$.
\end{theorem}\bigskip

When $\alpha<1/2$, the ``isotropic'' scaling by a factor $1/\sqrt{n}$ in both directions is no more adequate, and in fact, it can be shown that the order of magnitude of the parts of $\lambda$ under Schur-Weyl measures with parameter $\alpha<1/2$ is a $O(n^{1-\alpha})$, see \cite[\S6]{FM10}. On the other hand, the function $\Omega_{0}$ is also the limit shape of rescaled Young diagram under the so-called Plancherel measure, and in this case, a central limit theorem for the fluctuation
$$\sqrt{n}\,\Delta_{\lambda}(s)=\sqrt{n}\,(\lambda^{*}(s)-\Omega_{0}(s))$$
has been proved by S. Kerov, see \cite{Ker93} and \cite{IO02}. Our paper is concerned with the analogue of this result for Schur-Weyl measures with parameters $\alpha=1/2$ and $c>0$. In the  next sections, we shall start by presenting the tools commonly used for the asymptotic study of representations of symmetric groups.
\bigskip
\bigskip

\section{Observables of diagrams}\label{obs}
In the setting of random partitions, the algebra of polynomial functions\footnote{In the same setting, another fruitful approach relie on determinantal point processes, see \emph{e.g.} \cite{BO00} and \cite{BO05}.} on Young diagrams and its various bases play an extremely important role that can be compared to the one played by the moments of a real random variable. Hence, the \textbf{generating function} of a Young diagram $\lambda$ is defined by
$$G_{\lambda}(z)=\frac{\prod_{i=1}^{v-1}z-y_{i}}{\prod_{i=1}^{v}z-x_{i}},$$
where $x_{1}<y_{1}<x_{2}<y_{2}<\cdots<x_{v-1}<y_{v-1}<x_{v}$ are the \textbf{interlacing coordinates} of $\lambda$, that is to say, the sequences of local minima and local maxima of the function $s\mapsto \lambda(s)$. Alternatively, if $\sigma(s)=\frac{\lambda(s)-|s|}{2}$, then 
$$G_{\lambda}(z)=\frac{1}{z}\exp\left(-\int_{\R}\frac{\sigma'(s)}{z-s}\,ds\right).$$
This latter definition allows to consider the generating function of more general objects such as \textbf{continuous Young diagrams}, \emph{i.e.}, functions $s\mapsto \omega(s)$ that satisfy the two following properties:
\begin{enumerate}
\item $\omega$ is Lipschitz with constant $1$, \emph{i.e.}, $|\omega(s)-\omega(t)|\leq |s-t|$ for all $s,t$.
\item $\omega(s)=|s|$ for $s$ big enough.
\end{enumerate}
Given a (continuous) diagram $\lambda$, the coefficients $\tilh_{n}(\lambda)$ of the power serie $z^{-1}G_{\lambda}(z^{-1})$ generate an algebra $\obs$ of \textbf{observables of diagrams}. One may also consider as an algebraic basis for $\obs$ the coefficients $\tilp_{n}(\lambda)$ given by the formula
$$z^{-1}G_{\lambda}(z^{-1})=1+\sum_{n=1}^{\infty} \tilh_{n}(\lambda)\,z^n=\exp\left(\sum_{n=1}^{\infty} \frac{\tilp_{n}(\lambda)}{n}\,z^{n}\right).$$
For a true Young diagram $\lambda$, $\tilp_{n}(\lambda)=\sum_{i=1}^{v}(x_{i})^{n}-\sum_{i=1}^{v-1}(y_{i})^{n}$ is the $n$-th moment of the interlacing coordinates, whereas for a continuous diagram, one can give the following formula:
$$\tilp_{n}(\omega)=\int_{\R} \sigma''(s)\,s^{n}\,ds=\frac{1}{(n+1)(n+2)}\int_{\R}\sigma(s)\,s^{n+2}\,ds$$
Notice that $\tilh_{1}(\lambda)=\tilp_{1}(\lambda)=0$ for any (continuous) diagram $\lambda$; on the other hand, one can show that the other coefficients are algebraically independant (viewed as functions on diagrams). So, $$\obs=\C[\tilh_{2},\tilh_{3},\ldots]=\C[\tilp_{2},\tilp_{3},\ldots],$$ see \cite[\S2]{IO02}. The \textbf{weight grading} on this algebra is then defined by setting $\wt(\tilp_{k\geq 2})=k$, and it provides a filtration of algebra on $\obs$ that is well-behaved with respect to scaling of continuous diagrams. More precisely, if $\omega$ is a continuous diagram and if $\omega^{t}$ is the scaled diagram defined by
$$\omega^{t}(s)=\frac{\omega(st)}{t}$$
(see figure \ref{scaling}), then for any homogeneous observable $f$ of weight $k$, $f(\omega^{t})=t^{k}\,f(\omega)$. In particular, if $\lambda$ is a true Young diagram and if $\lambda^{*}=\lambda^{1/\sqrt{|\lambda|}}=\lambda^{1/\sqrt{n}}$, then $f(\lambda^{*})=n^{-\frac{k}{2}}\,f(\lambda)$.
\comment{\psset{unit=1mm}\pspicture(-50,0)(50,50)
\psline{->}(-50,0)(50,0)
\psline{->}(0,0)(0,50)
\parametricplot[border=2mm,bordercolor=white]{-2}{2}{t 20 mul 
                       t t 0.5 mul arcsin mul 0.017453 mul 4 t t mul neg add sqrt add 12.732395 mul t t mul t mul t mul t t mul 8 mul neg add 16 add 0.5 neg mul t mul add}
\parametricplot[border=2mm,bordercolor=white]{-2}{2}{t 20 mul 0.5 mul
                       t t 0.5 mul arcsin mul 0.017453 mul 4 t t mul neg add sqrt add 12.732395 mul t t mul t mul t mul t t mul 8 mul neg add 16 add 0.5 neg mul t mul add 0.5 mul}
\psline(-47,47)(-20,20)
\psline(47,47)(20,20)
\psline[linewidth=0.25pt](20,20)(0,0)(-20,20)
\rput(-35,39.5){$\omega$}
\rput(-13,21){$\omega^{1/2}$}
\endpspicture}

\figcap{\includegraphics{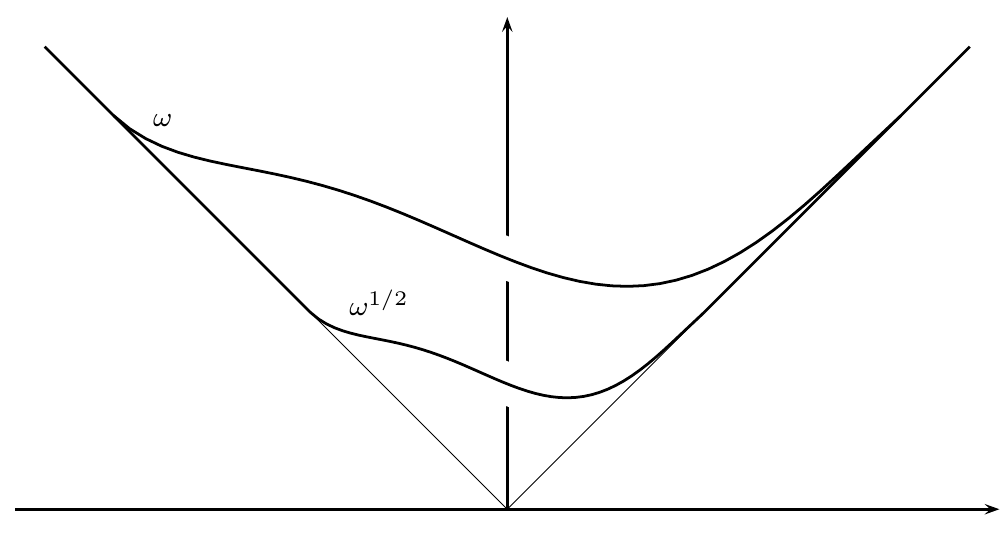}}{A continuous Young diagram $\omega$ and a rescaled version $\omega^{t}$. \label{scaling}}
\bigskip

Another algebraic basis of $\obs$ consists in the \textbf{free cumulants} of the transition measures of the continuous Young diagrams. First, if $\lambda$ is a (continuous) Young diagram, we define its \textbf{transition measure} $\mu_{\lambda}$ by
$$G_{\lambda}(z)=\int_{\R}\frac{\mu_{\lambda}(s)}{z-s}.$$
In other words, the generating function of $\lambda$ is the Stieltjes transform of the transition measure. Notice that the $\tilh_{n}(\lambda)$'s are precisely the moments of the compactly supported measure $\mu_{\lambda}$. Then, the free cumulants of the measure $\mu_{\lambda}$ are the coefficients of the formal inverse $R_{\lambda}$ of $G_{\lambda}$ in a neighboorhood of infinity:
$$R_{\lambda}(G_{\lambda}(z))=z \,\,\,\text{for }z\text{ big enough}\qquad;\qquad R_{\lambda}(z)=\frac{1}{z}\left(1+\sum_{n=1}^{\infty}R_{n}(\mu_\lambda)\,z^{n}\right)$$
see \cite{VDN92} for the uses of this notion in free probability. By Lagrange inversion, the free cumulants also satisfy the formula
$$R_{n+1}(\lambda)=R_{n+1}(\mu_{\lambda})=-\frac{1}{n}[z^{-1}](G_{\lambda}(z))^{-n},$$
and $R_{n+1}$ is therefore an observable of diagrams that is homogeneous with weight $n+1$. R. Speicher has given a simple combinatorial interpretation of the change of basis formulas between the $\tilh$'s and the $R$'s, see \cite{NS06}. These formulas implie in particular that the $R_{k\geq 2}$'s are algebraically independant, so $\obs=\C[R_{2},R_{3},\ldots]$.
\bigskip

Finally, another algebraic basis comes from the so-called \textbf{central characters}, that allow to relie the algebra $\obs$ to the representation theory of the symmetric groups. If $\lambda$ and $\mu$ are two partitions of same size $n$, one denotes by $\chi^{\lambda}(\mu)$ the value of the normalized irreducible character of $\sym_{n}$ of label $\lambda$ on a permutation $\sigma_{\mu}$ of cycle type $\mu$. Then, if $\lambda$ and $\mu$ are two partitions of sizes $n$ and $k$, one sets:
$$\varSigma_{\mu}(\lambda)=\begin{cases} n^{\downarrow k}\,\chi^{\lambda}(\mu\sqcup 1^{n-k})&\text{if }n\geq k,\\
0&\text{otherwise}.
\end{cases}$$
Here, $\mu\sqcup 1^{n-k}$ is the completed partition obtained by adding parts of size $1$ to $\mu$, and $n^{\downarrow k}$ is the falling factorial $n(n-1)\cdots(n-k+1)$. It is true, although totally non trivial, that the $\varSigma_{\mu}$'s are observables of diagrams; \emph{cf.} \cite[\S3-4]{IO02}. More precisely, $\wt(\varSigma_{\mu})=|\mu|+\ell(\mu)$, and the top homogeneous component of $\varSigma_{\mu}$ is 
$$R_{\mu+1}=R_{\mu_{1}+1}R_{\mu_{2}+1}\cdots R_{\mu_{r}+1},$$
see \cite{Bia98} and \cite[\S10]{IO02}. As a consequence, $\obs=\C[\varSigma_{1},\varSigma_{2},\ldots]$ is freely generated by the central characters of cycles, and one has a factorization property of characters in top homogeneous component:
$$\varSigma_{\mu_{1}}\,*\,\varSigma_{\mu_{2}}=\varSigma_{\mu_{1}\sqcup \mu_{2}}+\big(\text{observable of weight less than }\mu_{1}+\mu_{2}+\ell(\mu_{1})+\ell(\mu_{2})-2\big)$$
In fact, there exists multivariate polynomials with non-negative integer coefficients $K_{k}$ such that $\varSigma_{k}=K_{k}(R_{k+1},R_{k},\ldots,R_{2})$ for any $k \geq 1$; these \textbf{Kerov polynomials} are studied for instance in \cite{DFS08}. In \S\ref{core}, we will rather have to exprime the $\tilp$'s as polynomials in the $\varSigma$'s. We refer to \cite[Proposition 3.7]{IO02} for a proof of the following result:
\begin{proposition}[Change of basis formula between $\tilp$'s and $\varSigma$'s]\label{changeofbasis}
For $k \geq 2$, $\tilp_{k}$ is the top homogeneous component of weight $k$ of the observable
$$\sum_{\substack{\mu=1^{m_{1}}2^{m_{2}}\cdots s^{m_{s}}\\|\mu|+\ell(\mu)=k }} \frac{k^{\downarrow \ell(\mu)}}{\prod_{i\geq 1} m_{i}!}\prod_{i\geq 1}(\varSigma_{i})^{m_{i}}\,.$$
\end{proposition}
\begin{examples}
Let us write the first $\tilp_{k}$'s in the linear basis of central characters $(\varSigma_{\mu})_{\mu \in \bigsqcup_{n \in \N}\Part_{n}}$:
\begin{align*}
\tilp_{1}&=0\quad\qquad\,;\qquad\tilp_{2}=2\varSigma_{1}\\
\tilp_{3}&=3\varSigma_{2}\qquad;\qquad\tilp_{4}=4\varSigma_{3}+6\varSigma_{1,1}+2\varSigma_{1}\\
\tilp_{5}&=5\varSigma_{4}+20\varSigma_{2,1}+15\varSigma_{2}\\
\tilp_{6}&=6\varSigma_{5}+30\varSigma_{3,1}+15\varSigma_{2,2}+20\varSigma_{1,1,1}+30\varSigma_{1,1}+60\varSigma_{3}+2\varSigma_{1}
\end{align*}
Hence, one sees for instance that $\tilp_{6}$ is the top homogeneous component of $6\varSigma_{5}+30\varSigma_{3,1}+15\varSigma_{2,2}+20\varSigma_{1,1,1}$, and if one uses the factorization property of the central characters, one sees that these formulas agree with Proposition \ref{changeofbasis}.
\end{examples}
\bigskip
\bigskip

\section{Some computations around the limit shapes $\Omega_{c}$}\label{limitshape}
By using observables of diagrams, it is very easy to prove Theorem \ref{firstasymptotic}; let us recall briefly this proof. If $\lambda$ is a partition of size $n$ picked randomly according to a Schur-Weyl measure, then any observable of diagram $f \in \obs$ yields a random variable $f(\lambda)$. When $f=\varSigma_{\mu}$, the expectation of the random variable $\varSigma_{\mu}(\lambda)$ is easy to compute:
\begin{align*}\SWa[\varSigma_{\mu}]&=\begin{cases}n^{\downarrow |\mu|}\,\tr_{(\C^{N})^{\otimes n}}(\sigma_{\mu\sqcup 1^{n-|\mu|}}) &\text{if }n \geq |\mu|,\\
0&\text{otherwise},
\end{cases}\\
&\simeq_{n \to \infty} n^{|\mu|}\,N^{\ell(\mu)-|\mu|} \simeq c^{|\mu|-\ell(\mu)}\,n^{\frac{|\mu|+\ell(\mu)}{2}},
\end{align*}
assuming that $\alpha=1/2$ and $c>0$. In particular, $\SWa[\varSigma_{\mu}]$ is a $O\big(n^{\frac{\wt(\varSigma_{\mu})}{2}}\big)$, and since the $\varSigma_{\mu}$'s form a linear basis of $\obs$, one concludes that for any observable $f$,
$$\SWa[f]=O\left(n^{\frac{\wt(f)}{2}}\right).$$
Now, $\varSigma_{\mu}$ and $R_{\mu+1}$ have the same top homogeneous component, so if $R_{\mu}=R_{\mu_{1}}R_{\mu_{2}}\cdots R_{\mu_{r}}$, then
$$\SWa[R_{\mu}]\simeq c^{|\mu|-2\ell(\mu)}\,n^{\frac{|\mu|}{2}}\qquad;\qquad \SWa[R_{\mu}(\lambda^{*})] \simeq c^{|\mu|-2\ell(\mu)}$$
for any partition $\mu$ without part of size $1$. It implies the convergence in probability $R_{k}(\lambda^{*}) \to c^{k-2}$ for any $k \geq 2$, whence the existence of limit shapes $\Omega_{c}$ with 
\begin{align*}R_{\Omega_{c}}(z)&=R_{c}(z)=\frac{1}{z}\left(1+\sum_{n=2}^{\infty}c^{n-2}z^{n}\right)=\frac{1}{z}+\frac{z}{1-cz}\\
G_{\Omega_{c}}(z)&=G_{c}(z)=\frac{c+z-\sqrt{(z-c)^{2}-4}}{2(1+cz)}=\frac{2}{z+c+\sqrt{(z-c)^{2}-4}}
\end{align*}
where the holomorphic square root is defined on $\C \setminus \R^{-}$ and chosen so that $\sqrt{1}=1$.
\bigskip

A careful analysis shows that these generating functions are indeed those of the continuous diagrams of Theorem \ref{firstasymptotic}, \emph{cf.} \cite[\S3]{Bia01}. In the following, we are rather interested in the computation of the transition measures $\mu_{c}=\mu_{\Omega_{c}}$ and of the observables $\tilh_{n}(\Omega_{c})$ and $\tilp_{n}(\Omega_{c})$. Let us begin with the transition measures:

\comment{\psset{unit=0.5mm}
\pspicture(-50,-73)(170,53)
\rput(-40,5){\textcolor{BurntOrange}{$c=0$}}
\rput(-40,-65){\textcolor{violet}{$c=1$}}
\rput(80,5){\textcolor{red}{$c=1/2$}}
\rput(80,-65){\textcolor{blue}{$c=2$}}
\parametricplot*[linecolor=BurntOrange!50!white]{-2}{2}{t 15 mul
4 t t mul -1 mul add sqrt 15 mul}
\parametricplot[linecolor=BurntOrange]{-2}{2}{t 15 mul
4 t t mul -1 mul add sqrt 15 mul}
\parametricplot*[linecolor=red!50!white]{-1.5}{2.5}{t 15 mul 120 add
4 t -0.5 add t -0.5 add mul -1 mul add sqrt 15 mul 1 0.5 t mul add div}
\parametricplot[linecolor=red]{-1.5}{2.5}{t 15 mul 120 add
4 t -0.5 add t -0.5 add mul -1 mul add sqrt 15 mul 1 0.5 t mul add div}
\pscustom[linecolor=violet!50!white,fillstyle=solid,fillcolor=violet!50!white]{
\parametricplot{-0.80}{3}{t 15 mul
4 t -1 add t -1 add mul -1 mul add sqrt 15 mul 1 t  add div -70 add}
\psline(45,-70)(-15,-70)}
\psframe*[linecolor=violet!50!white](-15,-70)(-12,-5)
\parametricplot[linecolor=violet]{-0.80}{3}{t 15 mul
4 t -1 add t -1 add mul -1 mul add sqrt 15 mul 1 t  add div -70 add}
\psline[linecolor=violet](-15,-70)(-15,-5)
\parametricplot*[linecolor=blue!50!white]{0}{4}{t 15 mul 120 add
4 t -2 add t -2 add mul -1 mul add sqrt 15 mul 1 t 2 mul add div -70 add}
\parametricplot[linecolor=blue]{0}{4}{t 15 mul 120 add
4 t -2 add t -2 add mul -1 mul add sqrt 15 mul 1 t 2 mul add div -70 add}
\psline[linecolor=blue,linewidth=2pt](112.5,-70)(112.5,-5)
\psline{->}(-50,0)(50,0)
\psline{->}(-50,-70)(50,-70)
\psline{->}(70,0)(170,0)
\psline{->}(70,-70)(200,-70)
\psline{->}(0,0)(0,50)
\psline{->}(0,-70)(0,-20)
\psline{->}(120,0)(120,50)
\psline{->}(120,-70)(120,-20)
\endpspicture
}
\figcap{\includegraphics{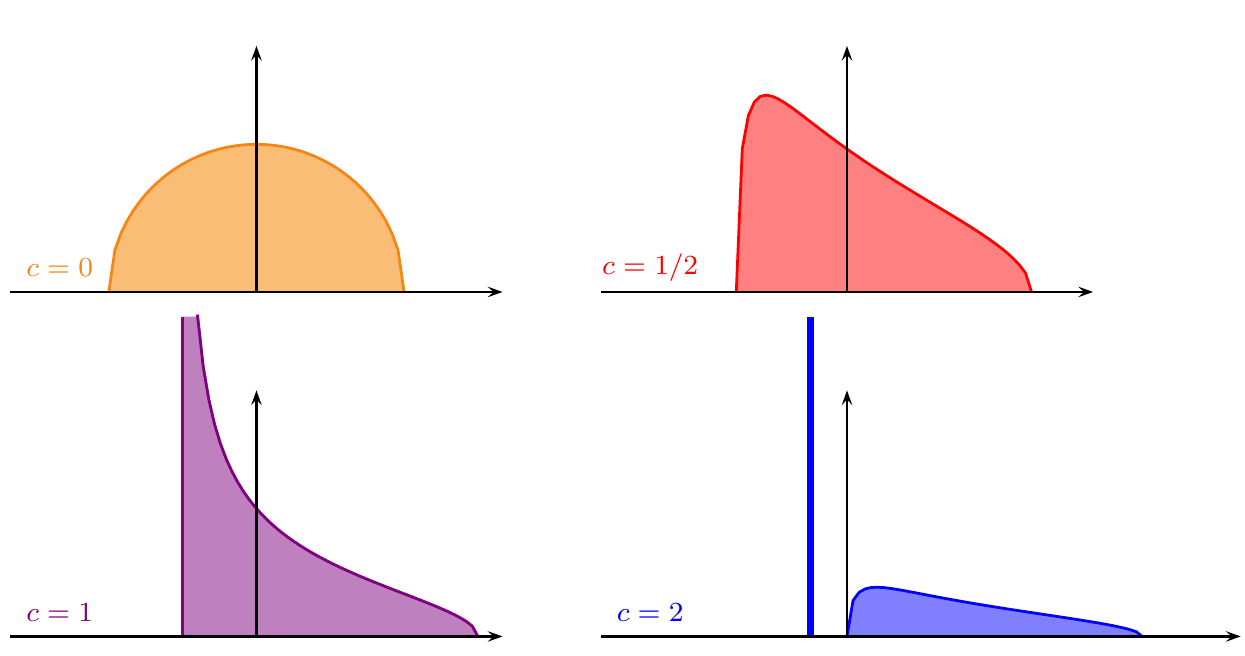}}{Deformations of the Wigner semicircle law corresponding to Schur-Weyl measures with parameters $\alpha=1/2$, $c>0$.}

\begin{proposition}[Limit shapes under Schur-Weyl measures and Mar\v{c}enko-Pastur distributions]\label{deformwigner}
The transition measure $\mu_{c}$ is up to an homothetic transformation the Mar\v{c}enko-Pastur distribution of parameter $c$:
\begin{align*}d\mu_{c\leq 1}(s)&=\mathbb{1}_{s \in [c-2,c+2]}\,\frac{\sqrt{4-(s-c)^{2}}}{2\pi(1+cs)}\,ds\\
d\mu_{c> 1}(s)&=\mathbb{1}_{s \in [c-2,c+2]}\,\frac{\sqrt{4-(s-c)^{2}}}{2\pi(1+cs)}\,ds+\left(1-\frac{1}{c^{2}}\right)\delta_{-\frac{1}{c}}(s)\end{align*}
Notice that one recovers the Wigner semicircle law when $c=0$.
\end{proposition}
\noindent These Mar\v{c}enko-Pastur distributions are in fact well-known, since they play the same role for Wishart ensembles as the Wigner semicircle law for the GUE, see \cite{MP67}.\medskip

\begin{proof}
One uses the Perron-Stieltjes inversion formula:
$$\frac{d\mu_{c}(s)}{ds}=\lim_{\eps \to 0^{+}}\frac{G_{c}(s-\I\eps)-G_{c}(s+\I\eps)}{2\I\pi}$$
which comes essentially from the calculus of residue. This formula holds for any real number $s$ such that $G_{c}$ stays bounded in the vicinity of $s$. Suppose first that $c< 1$. Then, the singularity of $G_{c}(z)$ at $z=-1/c$ is removable, so the Perron-Stieltjes formula holds for any $s$. If $s \in\, ]c-2,c+2[$, then $(s-c)^{2}-4$ is a negative real number, so
$$\lim_{\eps \to 0^{\pm}}\sqrt{(s+\I\eps)^{2}-4}=\pm \I\sqrt{4-(s-c)^{2}}.$$
The other terms in $G_{c}(z)$ are continuous at $z=s$, so 
$$\frac{d\mu_{c}(s \in \,]c-2,c+2[)}{ds}=\frac{(\I\sqrt{4-(s-c)^{2}}) - (-\I\sqrt{4-(s-c)^{2}})}{4\I\pi(1+cs)}=\frac{\sqrt{4-(s-c)^{2}}}{2\pi(1+cs)}.$$
If $|s-c|\geq 2$, then $(s-c)^{2}-4$ is a non-negative real number, so the square root is continuous in a vicinity of $s$. Consequently, 
$$\frac{d\mu_{c}(s \notin \,]c-2,c+2[)}{ds}=0,$$
whence the formula for $\mu_{c<1}$. Now, if $c\geq 1$, then the Perron-Stieltjes formula do not hold anymore at $s=-1/c$, and one has to add a Dirac to take account of the singularity of $G_{c}(z)$ at this point. Since the residue of $G_{c}(z)$ at $z=-1/c$ is 
$$\frac{-1/c+c-\sqrt{(c+1/c)^{2}-4}}{2c}=\frac{-(c-1/c)-\sqrt{(c-1/c)^{2}}}{2c}=-\left(1-\frac{1}{c^{2}}\right),$$
one obtains indeed the formula given by Proposition \ref{deformwigner}.
\end{proof}\bigskip

Now, let us compute the observables $\tilh_{n}(\Omega_{c})$ and $\tilp_{n}(\Omega_{c})$. Of course, it amounts to expand in power series the functions 
$$H_{c}(z)=z^{-1}G_{c}(z^{-1})=\frac{2}{1+cz+\sqrt{(1-cz)^{2}-4z^{2}}}$$
and $P_{c}(z)=zH_{c}'(z)/H_{c}(z)$, but the calculations are not at all trivial, and we shall need in particular the following hypergeometric identity:
\begin{lemma}\label{whatever}
For any non-negative integer $m$ and $k$,
$$\sum_{l=0}^{k}\sum_{u=0}^{m} \frac{(2k+2+u-2l)\,\,m+2l+2-u!\,2k-2l+u!}{(m-u+l+2)(m-u+l+1)\,m-u!\,u!\,l!\,l+1!\,k-l!\,k-l+1!}=  \frac{m+2k+4}{m+k+2}\,\times\,\frac{m+2k+2!}{m!\, k!\,k+2!}.$$
\end{lemma}\medskip

\begin{proof}
For such (multivariate) hypergeometric identities, one can use the theory of Wilf-Zeilberger, see \cite{PWZ97}, \cite{WZ92a} and \cite{WZ92b}. More precisely, there is an algorithm that provides recurrence relations satisfied by both sides of the identity, and on the other hand, it is easy to show that the identity holds for certain values of $m$ and $k$. Consequently, the identity is indeed true for all $m$ and $k$; we leave the details of the computation of the recurrence relations to any computer algebra system.
\end{proof}\bigskip

\begin{proposition}[Observables of the limit shapes]\label{obslimitshape}
The $n$-th moment of the law $\mu_{c}$ is
$$\tilh_{n}(\Omega_{c})=\mu_{c}(s^{n})=\sum_{k=1}^{\lfloor \frac{n}{2}\rfloor} \frac{n^{\downarrow 2k}}{(n-k+1)(n-k)\,k!\,k-1!}\,c^{n-2k},$$
and the $n$-th moment of interlacing coordinates is
$$\tilp_{n}(\Omega_{c})=\sum_{k=1}^{\lfloor \frac{n}{2}\rfloor} \frac{n^{\downarrow 2k}}{(n-k)\,k!\,k-1!}\,c^{n-2k}.$$
\end{proposition} 
\noindent When $c=0$, one recovers the fact that the even moments of the semicircle law are the Catalan numbers: $\tilh_{2n}(\Omega_{0})=C_{n}$.
\begin{proof}
The expansion of $2/H_{c}(z)$ gives
\begin{align*}
\frac{2}{H_{c}(z)}&=1+cz+(1-cz)\sqrt{1-4\left(\frac{z}{1-cz}\right)^{2}}\\
&=1+cz+(1-cz)\left(1-2\sum_{m=0}^{\infty}C_{m}\left(\frac{z}{1-cz}\right)^{2(m+1)}\right)\\
&=2\left(1-\sum_{m=0}^{\infty}C_{m}\,\frac{z^{2m+2}}{(1-cz)^{2m+1}}\right)
\end{align*}
because $\sum_{m=0}^{\infty}C_{m}\,x^{m}=\frac{1-\sqrt{1-4x}}{2x}$. Then,
$$H_{c}(z)=1+\sum_{m=0}^{\infty}\sum_{r=1}^{\infty}\left\{\sum_{m_{1}+m_{2}+\cdots+m_{r}=m}\!\!\!\!\! C_{m_{1}}C_{m_{2}}\cdots C_{m_{r}}\right\}\frac{z^{2m+2r}}{(1-cz)^{2m+r}}.
$$
Let us denote by $f(m,r)$ the sum between brackets. One has $f(m,0)=0$, $f(m,1)=C_{m}$ and $f(m,r)=f(m+1,r-1)-f(m+1,r-2)$ because of the recurrence relation satisfied by Catalan numbers. Consequently, one can show by induction that
$$f(m,r)=r\,\frac{2m+r-1!}{m!\,m+r!}.$$
Finally, one can use the expansion $\frac{1}{(1-cz)^{2m+r}}=\sum_{l=0}^{\infty} \binom{2m+r+l-1}{l}\,c^{l}z^{l}$ to obtain:
\begin{align*}H_{c}(z)&=1+\sum_{m=0}^{\infty}\sum_{r=1}^{\infty}\sum_{l=0}^{\infty} r\frac{2m+r+l-1!}{l!\,2m+r-1!}\,\frac{2m+r-1!}{m!\,m+r!}\,c^{l}\,z^{2m+2r+l}\\
&=1+\sum_{n=2}^{\infty}z^{n}\left\{\sum_{k=1}^{\lfloor \frac{n}{2}\rfloor }\left(\sum_{r=1}^{k}r\, \frac{(n-r-1)^{\downarrow 2k-r-1}}{k!\,k-r!}\right)c^{n-2k}\right\}
\end{align*}
For any positive integers $a$, $b$ and $c$, $\binom{a+1}{b+c+1}=\sum_{j\geq 1} \binom{j}{b}\binom{a-j}{c}$ --- this identity can be obtained combinatorially by grouping the parts of size $b+c+1$ of $\lle 1,a+1\rre$ according to the value of their $(b+1)$-th element. In particular,
$$\sum_{j\geq 1}\,j\binom{n-j-1}{n-k-1}=\binom{n}{n-k+1}$$
if one takes $a=n-1$, $b=1$ and $c=n-k-1$. As a consequence, the term between parentheses in the previous expression is just $\frac{n^{\downarrow 2k}}{(n-k+1)(n-k)\,k!\,k-1!}$, whence the formula for the $\tilh_{n}(\Omega_{c})$'s.\bigskip

As for the $\tilp_{n}(\Omega_{c})$'s, one can use Newton relations between the power sums and the complete homogeneous functions in the algebra of symmetric functions, or expand in power serie the function $P_{c}(z)$ --- it leads to the same computations. First, one sees that:
\begin{align*}P_{c}(z)&=-\frac{z}{1+cz+\sqrt{(1-cz)^{2}-4z^{2}}}\left(c - \frac{c(1-cz)+4z}{\sqrt{(1-cz)^{2}-4z^{2}}}\right)\\
&=\frac{z\,H_{c}(z)}{2}\left(\frac{c(1-cz)+4z}{\sqrt{(1-cz)^{2}-4z^{2}}}-c\right)\end{align*}
\begin{align*}
&=\frac{z\,H_{c}(z)}{2}\left(c\sum_{n=1}^{\infty}\binom{2n}{n}\left(\frac{z}{1-cz}\right)^{2n}+4\sum_{n=0}^{\infty}\binom{2n}{n}\left(\frac{z}{1-cz}\right)^{2n+1}\right)\\
&=\frac{z\,H_{c}(z)}{2}\left(\sum_{n=0}^{\infty}\sum_{r=0}^{\infty}\left\{cz\binom{2n+1+r}{r}\binom{2n+2}{n+1}+ 4\binom{2n+r}{r}\binom{2n}{n}\right\}c^{r}z^{2n+1+r}\right)
\end{align*}
If $n_{1}\geq 2$, then the coefficient of $z^{n_{1}}$ in $zH_{c}(z)$ is $\tilh_{n-1}(\Omega_{c})$, that is to say:
$$ \sum_{k=1}^{\lfloor \frac{n_{1}-1}{2}\rfloor} \frac{n_{1}-1!}{(n_{1}-k)(n_{1}-k-1)\,k!\,k-1!\,n_{1}-1-2k!}\,c^{n_{1}-2k-1}$$
On the other hand, if $n_{2} \geq 1$, then the coefficient of $z^{n_{2}}$ in the second term $T_{c}(z)$ of the product $P_{c}(z)$ is: 
$$\sum_{l=0}^{\lfloor \frac{n_{2}}{2}\rfloor-1}\!\!\!\!\frac{n_{2}-1!}{n_{2}-2l-2!\,l!\,l+1!}\,c^{n_{2}-2l-1}+\,\,2\!\!\sum_{l=0}^{\lfloor\frac{n_{2}-1}{2} \rfloor}\!\!\!\!\frac{n_{2}-1!}{n_{2}-2l-1!\,l!\,l!}\,c^{n_{2}-2l-1}=\frac{1}{n_{2}}\!\!\sum_{l=0}^{\lfloor\frac{n_{2}-1}{2}\rfloor}\!\!\!\! \frac{n_{2}+1!}{n_{2}-2l-1!\,l!\,l+1!}\,c^{n_{2}-2l-1}$$
One also has to take account of the coefficient of $z$ in $zH_{c}(z)$, that is equal to $1$. So, for $n \geq 3$, the coefficient of $z^{n}$ in $P_{c}(z)$ is equal to the sum of the following expressions:
\begin{align*}
A&=\frac{1}{n-1}\sum_{k=1}^{\lfloor\frac{n}{2} \rfloor} \frac{n!}{n-2k!\,k!\,k-1!}\,c^{n-2k}\\
B&=\!\!\!\!\!\!\sum_{\substack{n_{1}+n_{2}=n\\n_{1}\geq 2,\,\,n_{2}\geq 1\\2\leq 2k_{1}\leq n_{1}-1\\ 0 \leq 2k_{2} \leq n_{2}-1}}\frac{(n_{2}+1)\,n_{1}-1!\,n_{2}-1!\,n_{1}-k_{1}-2!}{n_{1}-k_{1}!\,k_{1}!\,k_{1}-1!\, n_{1}-2k_{1}-1!\,n_{2}-2k_{2}-1!\,k_{2}!\,k_{2}+1!}\,c^{n-2k_{1}-2k_{2}-2}\\
&=\!\!\!\!\!\!\sum_{\substack{n_{1}+n_{2}=n-3\\n_{1}\geq 0,\,\,n_{2}\geq 0\\0\leq 2k_{1}\leq n_{1}-1\\ 0 \leq 2k_{2} \leq n_{2}}}\frac{(n_{2}+2)\,n_{1}+1!\,n_{2}!\,n_{1}-k_{1}-1!}{n_{1}-k_{1}+1!\,k_{1}!\,k_{1}+1!\, n_{1}-2k_{1}-1!\,n_{2}-2k_{2}!\,k_{2}!\,k_{2}+1!}\,c^{n-2k_{1}-2k_{2}-4}\\
&=\!\!\!\!\!\!\sum_{\substack{0\leq u \leq n-4-2k\\0\leq k = k_{1}+k_{2}}} \frac{f(k_{1},k_{2},u,n)\,n-2k_{2}-2-u!\,2k_{2}+u!}{n-u-2k-4!\,u!\,k_{1}!\,k_{1}+1!\,k_{2}!\,k_{2}+1!}\,c^{n-2k-4}
\end{align*}
where $f(k_{1},k_{2},u,n)=\frac{2k_{2}+2+u}{(n-u-k_{1}-2k_{2}-2)(n-u-k_{1}-2k_{2}-3)}$. Consequently, $\tilp_{n}(\Omega_{c})$ is indeed a polynomial in $c$ with all terms of even degree or all terms of odd degree. The term of degree $n-2$ comes exclusively from $A$, and is 
$$\frac{1}{n-1}\,\frac{n!}{n-2!}=n=\frac{n^{\downarrow 2} }{n-1\,1!\,0!},$$
so the formula of Proposition \ref{obslimitshape} is true for the coefficient of $c^{n-2}$. For the other coefficients, one has to show that if $k\geq 0$ and $n-2k-4 \geq 0$, then
$$\sum_{\substack{0\leq u \leq n-4-2k\\ k = k_{1}+k_{2}}} \frac{f(k_{1},k_{2},u,n)\,n-2k_{2}-2-u!\,2k_{2}+u!}{n-u-2k-4!\,u!\,k_{1}!\,k_{1}+1!\,k_{2}!\,k_{2}+1!}=\frac{n^{\downarrow 2k+4}}{(n-k-2)(n-1)\,k+2!\,k!}.$$
Indeed, the left-hand side comes from $B$, and the right-hand side is the difference of the coefficient in the formula of Proposition \ref{obslimitshape} and of the coefficient coming from $A$. Up to a change of index, this last identity is exactly Lemma \ref{whatever}. Finally, it is easy to verify that the formula for the moments $\tilp_{n}(\Omega_{c})$ remains true for $n=1,2$.
\end{proof}
\bigskip

\noindent These calculations are of course unessential, but we shall need the precise expression of the moments $\tilp_{n}(\Omega_{c})$ in our study of the fluctuations $\Delta_{\lambda,c}(s)=\lambda^{*}(s)-\Omega_{c}(s)$.
\bigskip

\section{Gaussian concentration of measures on partitions}\label{sniady}
To establish the gaussian concentration of probability measures on partitions coming from reducible representations of symmetric groups, P. \'Sniady has developed in \cite{Sni06} a theory of cumulants of observables that works in a very general setting, and in particular for Schur-Weyl measures. If $X_{1},\ldots,X_{r}$ are real or complex random variables, we recall that the \textbf{joint cumulant} of $X_{1},\ldots,X_{r}$ is
$$k(X_{1},\ldots,X_{r})=\left.\frac{\partial^{r}}{\partial t_{1}\cdots \partial t_{r}}\right|_{t_{1}=t_{2}=\cdots=t_{r}=0}\!\!\!\log \esper[\exp(t_{1}X_{1}+\cdots+ t_{r}X_{r})]\,.$$
In particular, gaussian vectors are characterized by the vanishing of joint cumulants of order $r\geq 3$, and the first and second joint cumulants give the expectation of the gaussian vector and its covariance matrix. That said, if $a_{1},\ldots,a_{r}$ commute in the group algebra $\C\sym_{n}$, and if $\proba$ is a probability measure on $\Part_{n}$, then one can consider the $a_{i}$'s as commutative random variables by setting
$$a_{i}=a_{i}(\lambda)=\chi^{\lambda}(a_{i})$$
with $\lambda$ picked randomly according to $\proba$; as a consequence, the joint cumulant $k(a_{1},\ldots,a_{r})$ makes sense.  On the other hand, we have seen that observables of diagrams may also be considered as random variables, and besides, any observable can be interpreted as an element of the center of the group algebra $\C\sym_{n}$, see \cite[\S2.1]{Sni06}; so, the joint cumulant of observables of diagrams also makes sense. Then, the major result of \cite{Sni06} is the following:
\begin{theorem}[Sniady's theory of cumulants of observables]\label{sniadytheo}
Let $(\proba_{n})_{n \in \N}$ be a family a probability measures on the sets $\Part_{n}$ of integer partitions, and let us denote by $(\esper_{n})_{n \in \N}$ and $(k_{n})_{n \in \N}$ the corresponding expectations and joint cumulants. The following assertions are equivalent:
\begin{enumerate}
\item For all positive integers $l_{1},\ldots,l_{r}$, 
$$k_{n}(\varSigma_{l_{1}},\ldots,\varSigma_{l_{r}})\,n^{-\frac{l_{1}+\cdots+l_{r}-r+2}{2}}=O(1).$$
\item For all integers $l_{1},\ldots,l_{r} \geq 2$,
$$k_{n}(R_{l_{1}},\ldots,R_{l_{r}})\,n^{-\frac{l_{1}+\cdots+l_{r}-2r+2}{2}}=O(1).$$
\item If $\sigma_{l_{1}}, \ldots,\sigma_{l_{r}}$ are disjoint cycles of respective lengths $l_{1},\ldots,l_{r}$, then $$k_{n}(\sigma_{l_{1}},\ldots,\sigma_{l_{r}})\,n^{\frac{l_{1}+\cdots+l_{r}+r-2}{2}}=O(1).$$
\end{enumerate}
Moreover, if these assertions hold, then the following limits are equal (assuming that they exist):
\begin{align*}c_{l+1}&=\lim_{n\to \infty} \esper_{n}[\varSigma_{l}]\,n^{-\frac{l+1}{2}}=\lim_{n\to \infty} \esper_{n}[R_{l+1}]\,n^{-\frac{l+1}{2}}=\lim_{n \to \infty} \esper_{n}[\sigma_{l}]\,n^{\frac{l-1}{2}}\\
v_{l+1,m+1}&=\lim_{n \to \infty}k_{n}(\varSigma_{l},\varSigma_{m})\,n^{-\frac{l+m}{2}}=\lim_{n \to \infty}k_{n}(R_{l+1},R_{m+1})\,n^{-\frac{l+m}{2}}\\
&=\lim_{n \to \infty}k_{n}(\sigma_{l},\sigma_{m})\,n^{\frac{l+m}{2}}-lm\,c_{l+1}\,c_{m+1}+\sum_{\substack{l=a_{1}+\cdots+a_{r}\\ m=b_{1}+\cdots + b_{r}}} \frac{lm}{r}\,c_{a_{1}+b_{1}}\cdots c_{a_{r}+b_{r}}
\end{align*} 
In this setting, the random processes
$$\left(n^{-\frac{l}{2}}\,\big(R_{l+1}-\esper_{n}[R_{l+1}]\big)\right)_{l\geq 1}\quad\text{and}\quad\left(n^{-\frac{l}{2}}\,\big({\varSigma_{l}}-\esper_{n}[\varSigma_{l}]\big)\right)_{l\geq 1}$$
converge in finite-dimensional laws towards a gaussian process whose covariance matrix is $(v_{l+1,m+1})_{l,m \geq 1}$.
\end{theorem}
\bigskip

Under a Schur-Weyl measure $\SWa$ with parameters $\alpha=1/2$ and $c>0$, the expectation of a permutation of cycle type $\mu$ is $\SWa[\sigma_{\mu}]=N^{\ell(\mu)-|\mu|}\simeq c^{|\mu|-\ell(\mu)}\,n^{\frac{\ell(\mu)-|\mu|}{2}}$, and this expression is multiplicative with respect to the lengths of the cycles of $\sigma_{\mu}$. Consequently, given disjoint cycles of lengths $l_{1},\ldots,l_{r}$, one has:
$$k_{n}(\sigma_{l_{1}},\ldots,\sigma_{l_{r}})=\begin{cases}N^{1-l_{1}}&\text{if }r=1,\\ 0&\text{otherwise}.\end{cases}$$
From this, one deduces that Schur-Weyl measures with parameters $\alpha=1/2$ and $c>0$ satisfy the hypotheses of Theorem \ref{sniadytheo}, and the limit of the random process
$$\left(X_{l}=\frac{\varSigma_{l}}{n^{\frac{l}{2}}}-c^{l-1}n^{\frac{1}{2}}\right)_{l \geq 2}$$
is a centered gaussian process with covariance matrix
$$v_{l+1,m+1}=0-lm\,c^{l+m-2}+\sum_{r\geq 1}\sum_{\substack{l=a_{1}+\cdots+a_{r}\\ m=b_{1}+\cdots + b_{r}}} \frac{lm}{r} \,c^{l+m-2r}=\sum_{r\geq 2} \binom{l}{r}\binom{m}{r}\,r\,c^{l+m-2r},$$
see \cite[Example 6]{Sni06}. In particular, when $c=0$, the random variables $\sqrt{n}\,\chi^{\lambda}(\sigma_{l})$ converge in finite-dimensional laws towards independent normal variables of variance $l$; this is a form of Kerov's central limit theorem. As we shall see in \S\ref{ksw}, the extension of Kerov's theorem to the case of Schur-Weyl measures stems from a very natural idea; namely, one wanted to produce \emph{independent} gaussian variables from the fluctuations ${\varSigma_{l}}\,{n^{-\frac{l}{2}}}-c^{l-1}n^{\frac{1}{2}}$ when $c$ is not equal to $0$ and the fluctuations are not asymptotically independent. The obvious way to do this was to perform a Gram-Schmidt orthogonalisation in the gaussian space of the limiting process, and we discovered that the orthogonalized basis is related to translated versions of the Chebyshev polynomials of the second kind.
\bigskip
\bigskip

\section{Moments of the deviations and Chebyshev polynomials of the second kind}\label{core}
To make appear the Chebyshev polynomials of the second type, we start by computing the moments of the deviation $\Delta_{\lambda,c}$ of a rescaled Young diagram from its limit shape. If $\lambda$ is a partition of size $n$, let us denote by $\tilq_{k,c}(\lambda)$ the following quantity:
$$\widetilde{q}_{k,c}(\lambda)=\frac{\widetilde{p}_{k+1}(\lambda)}{(k+1)\,n^{k/2}}-\sum_{l=1}^{\lfloor\frac{k+1}{2}\rfloor} \frac{k^{\downarrow 2l-1}}{(k+1-l)\,l!\,l-1!}\,c^{k+1-2l}\,n^{1/2}$$
Because of the factors $n^{-k/2}$ and $n^{1/2}$, $\widetilde{q}_{k,c}$ is not contained in $\obs$, but is rather in the localized ring $\obs^{+}=\obs[{\varSigma_{1}}^{-1/2}]$. In \cite{IO02}, all the computations are done in this ring of generalized observables, and using the so-called Kerov degree; as we shall see here, this complicated framework is not at all required for the asymptotic analysis of fluctuations of Plancherel ($c=0$) and Schur-Weyl ($c>0$) measures.

\begin{lemma}\label{moment}
Let $\lambda$ be a partition of size $n$, and let us define the deviation $\Delta_{\lambda,c}$ by 
$\Delta_{\lambda,c}(s)=\lambda^{*}(s)-\Omega_{c}(s)$. Then, for any $k\geq 1$,
$$\frac{\sqrt{n}}{2}\int_{\R}s^{k}\,\Delta_{\lambda,c}(s)\,ds=\frac{\tilq_{k+1,c}(\lambda)}{k+1}.$$
\end{lemma}
\medskip

\begin{proof}
Notice that this lemma is the exact analogue of \cite[Proposition 7.2]{IO02}. Besides, the proof is exactly the same. Hence, by using the second part of Proposition \ref{obslimitshape}, one gets:
$$\widetilde{q}_{k,c}(\lambda)=\frac{\sqrt{n}}{k+1}\left(\frac{\tilp_{k+1}(\lambda)}{n^{\frac{k+1}{2}}}-\tilp_{k+1}(\Omega_{c})\right)=\sqrt{n}\left(\frac{{\tilp_{k+1}(\lambda^{*})}-\tilp_{k+1}(\Omega_{c})}{k+1}\right)$$
Then, $\Delta_{\lambda,c}(s)=\lambda^{*}(s)-\Omega_{c}(s)=(\lambda^{*}(s)-|s|)-(\Omega_{c}(s)-|s|)$, so:
\begin{align*}
\frac{\sqrt{n}}{2}\int_{\R}s^{k}\,\Delta_{\lambda,c}(s)\,ds&=\sqrt{n} \left(\int_{\R}s^{k}\,\sigma_{\lambda^{*}}(s)\,ds-\int_{\R}s^{k}\,\sigma_{\Omega_{c}}(s)\,ds\right)\\
&=\frac{\sqrt{n}}{(k+1)(k+2)}\left(\int_{\R}s^{k+2}\,\sigma_{\lambda^{*}}''(s)\,ds-\int_{\R}s^{k+2}\,\sigma_{\Omega_{c}}''(s)\,ds\right)\\
&=\frac{\sqrt{n}}{(k+1)(k+2)}\big(\tilp_{k+2}(\lambda^{*})-\tilp_{k+2}(\Omega_{c})\big)=\frac{\tilq_{k+1,c}(\lambda)}{k+1}
\end{align*}
\end{proof}\bigskip

\begin{proposition}[Asymptotic behaviour of linear functionals of the deviation]\label{linearfunctional}
Let $p(s)$ be a polynomial in $s$. One chooses $\lambda\in \Part_{n}$ randomly according to the Schur-Weyl measure with parameters $\alpha=1/2$ and $c>0$. Then, 
$$\sqrt{n}\scal{p(s)}{\Delta_{\lambda,c}(s)}=\sqrt{n}\int_{\R}p(s)\,\Delta_{\lambda,c}(s)\,ds$$ converges in law towards a centered (gaussian) random variable.
\end{proposition}

\begin{proof}
The important part of Proposition \ref{linearfunctional} is actually the fact that the limit random variable exists and is centered, see the proof of Theorem \ref{almosttheorem} hereafter. Because of Theorem \ref{sniadytheo}, for any $l\geq 2$, 
$$\sqrt{n}\left(\frac{R_{l}(\lambda)}{n^{\frac{l}{2}}}-c^{l-2}\right)=\sqrt{n}\left(R_{l}(\lambda^{*})-R_{l}(\Omega_{c})\right)$$
converges towards a centered gaussian variable, and the result is also true when $l=1$ (in this case the variables are null for all $n$). Now, if $\mu=(\mu_{1},\ldots,\mu_{r})$ is a partition, then
$$\sqrt{n}\left(R_{\mu}(\lambda^{*})-R_{\mu}(\Omega_{c})\right)=\sum_{i=0}^{r-1}R_{\mu_{1}\ldots\mu_{r-i-1}}(\lambda^{*})\,\left[\sqrt{n}\,(R_{\mu_{r-i}}(\lambda^{*})-R_{\mu_{r-i}}(\Omega_{c}))\right]\,R_{\mu_{r-i+1}\ldots\mu_{r}}(\Omega_{c}).$$
Because of Theorem \ref{firstasymptotic}, each $R_{\mu_{1}\ldots\mu_{r-i-1}}(\lambda^{*})$ converges towards the constant variable $R_{\mu_{1}\ldots\mu_{r-i-1}}(\Omega_{c})$, and on the other hand, each difference $\left[\sqrt{n}\,(R_{\mu_{r-i}}(\lambda^{*})-R_{\mu_{r-i}}(\Omega_{c}))\right]$ converges towards a centered gaussian variable. Moreover, for these differences of free cumulants $\sqrt{n}\,(R_{l}(\lambda^{*})-R_{l}(\Omega_{c}))$, one has in fact joint convergence in finite-dimensional laws towards a centered gaussian process. Consequently, the limit of $\sqrt{n}(R_{\mu}(\lambda^{*})-R_{\mu}(\Omega_{c}))$ is an element of a gaussian space whose variables are all centered:
$$\forall \mu,\,\,\,\Delta_{n,c} (R_{\mu})=\sqrt{n}\,(R_{\mu}(\lambda^{*})-R_{\mu}(\Omega_{c})) \quad\longrightarrow \quad\mathcal{N}(0,(\sigma_{\mu,c})^{2})$$
In fact, the same argument as before shows that given a vector $(\Delta_{n,c}(R_{\mu^{(1)}}), \ldots, \Delta_{n,c}(R_{\mu^{(s)}}))$ of such fluctuations, one has a joint convergence towards a gaussian vector. Since the $R_{\mu}$'s form a linear basis of $\obs$, the same result holds for any vector $(\Delta_{n,c}(f^{(1)}),\ldots,\Delta_{n,c}(f^{(s)}))$ of scaled fluctuations of observables of diagrams: it converges in law towards a centered gaussian vector. As Lemma \ref{moment} ensures that
$$\frac{\sqrt{n}}{2}\scal{s^{k}}{\Delta_{\lambda,c}(s)}=\frac{\Delta_{n,c}(\tilp_{k+2})}{(k+1)(k+2)},$$
 our claim is proved.
\end{proof}\bigskip

\begin{lemma}\label{transmoment}
With the same notations as in Lemma \ref{moment},
\begin{align*}&\frac{\sqrt{n}}{2}\int_{\R}(s-c)^{k}\,\Delta_{\lambda,c}(s)\,ds=\\
&\frac{\sqrt{n}}{(k+1)(k+2)}\left(\sum_{2 \leq l \leq k+2}\!\! \binom{k+2}{l} (-c)^{k+2-l}\,\frac{\widetilde{p}_{l}(\lambda)}{n^{l/2}}-\sum_{2 \leq 2m \leq k+2}\!\! \binom{k+2}{m} (-c)^{k+2-2m}\right)
\end{align*}
\end{lemma}\medskip
\begin{proof}
Of course, one applies Newton's binomial theorem in order to expand the power $(s-c)^{k}$:
\begin{align*}
\frac{\sqrt{n}}{2}\scal{(s-c)^{k}}{\Delta_{\lambda,c}(s)}&=\sum_{l=0}^{k}\binom{k}{l}(-c)^{k-l}\,\frac{\sqrt{n}}{2}\int_{\R}s^{l}\,\Delta_{\lambda,c}(s)\,ds=\frac{1}{k+1}\sum_{l=0}^{k}\binom{k+1}{l+1}(-c)^{k-l}\,\tilq_{l+1,c}(\lambda)\\
&=\frac{1}{k+1}\sum_{l=1}^{k+1}\binom{k+1}{l}(-c)^{k+1-l}\,\tilq_{l,c}(\lambda)\\
&=\left(\frac{\sqrt{n}}{(k+1)(k+2)}\sum_{l=2}^{k+2}\binom{k+2}{l}(-c)^{k+2-l}\,\frac{\tilp_{l}(\lambda)}{n^{l/2}}\right)-\sqrt{n}\,A(c)
\end{align*}
where $A(c)$ is the following quantity (one just uses the definition of $\tilq_{l,c}(\lambda)$):
\begin{align*}
A(c)&=\sum_{m=1}^{\lfloor \frac{k}{2}\rfloor +1} \sum_{l=2m-1}^{k+1} \frac{k!}{k+1-l!\,l-2m+1!\,m!\,m-1!\, (l+1-m)}\,(-1)^{k+1-l}c^{k+2-2m} \\
&=\sum_{m=0}^{\lfloor \frac{k}{2}\rfloor} \frac{k!}{k-2m!\,m!\,m+1!}\left\{\sum_{u=0}^{k-2m} \binom{k-2m}{u}\,\frac{(-1)^{k-u-2m}}{u+m+1}\right\}c^{k-2m}
\end{align*}
The term between brackets can be computed by writing $\frac{1}{u+m+1}=\int_{0}^{1} x^{u+m}\,dx$. So,
\begin{align*}
\sum_{u=0}^{k-2m} \binom{k-2m}{u}\,\frac{(-1)^{k-u-2m}}{u+m+1}&=\int_{0}^{1} \sum_{u=0}^{k-2m} \binom{k-2m}{u}\,(-1)^{k-u-2m}\,x^{u+m}\,dx\\
&=\int_{0}^{1}x^{m}\,(x-1)^{k-2m}\,dx=(-1)^{k-2m}\,\frac{k-2m!\,m!}{k-m+1!}
\end{align*}
and consequently, $A(c)=\sum_{m=0}^{\lfloor \frac{k}{2}\rfloor} \frac{k!}{k-m+1!\,m+1!}\,(-c)^{k-2m}=\frac{1}{(k+1)(k+2)}\sum_{m=0}^{\lfloor \frac{k}{2}\rfloor} \binom{k+2}{m+1}\,(-c)^{k-2m}$, whence the formula announced.
\end{proof}
\bigskip

In the following, we denote by $U_{k}(s)$ the Chebyshev polynomial of the second kind renormalized so that $U_{k}(2\cos\theta)=\frac{\sin (k+1)\theta}{\sin \theta}$. These polynomials satisfy the relation
$$U_{k+2}(X)=X\,U_{k+1}(X)-U_{k}(X),$$
and the first values are $U_{0}(X)=0$, $U_{1}(X)=X$, $U_{2}(X)=X^{2}-1$, $U_{3}(X)=X^{3}-2X$ and $U_{4}(X)=X^{4}-3X^{2}+1$. The end of this paragraph is devoted to the proof of the following theorem:
\begin{theorem}[Linear functionals of the deviation associated to translated Chebyshev polynomials]\label{almosttheorem}
For any non-negative integer $k$, $\frac{\sqrt{n}}{2}\scal{U_{k}(s-c)}{\Delta_{\lambda}(s)}$ is equal to
$$\frac{1}{k+1}\sum_{l=0}^{k-1}\binom{k+1}{l}(-c)^{l}\,X_{k+1-l}(\lambda),$$
plus some observables that  under Schur-Weyl measures of parameter $1/2$ converge towards $0$ when $n$ goes to infinity.
\end{theorem}
\begin{examples}
Since the proof is quite tricky, we have found useful to include in our paper the complete computations for $k=0,1,2,3,4$ (one has to go up to $k=4$ to see all the tricks). In particular, this will make Lemma \ref{nasty} much more natural. One will use the relations between the $\tilp$'s and the $\varSigma$'s presented at the end of \S\ref{obs}.
\begin{enumerate}
\item[$k=0$.] Since $\Delta_{\lambda,c}(s)$ is the difference between two normalized continuous Young diagrams, $$\scal{U_{0}(s-c)}{\Delta_{\lambda,c}(s)}=\scal{1}{\Delta_{\lambda,c}(s)}=0,$$
and this agrees with the formula of Theorem \ref{almosttheorem}.\vspace{2mm}
\item[$k=1$.] $\scal{U_{1}(s-c)}{\Delta_{\lambda,c}(s)}=\scal{s-c}{\Delta_{\lambda,c}(s)}=\scal{s}{\Delta_{\lambda,c}(s)}$ can be computed by using Lemma \ref{moment}. Thus,
$$\frac{\sqrt{n}}{2}\scal{U_{1}(s-c)}{\Delta_{\lambda,c}(s)}=\frac{\widetilde{q}_{2}(\lambda)}{2}=\frac{\widetilde{p}_{3}(\lambda)}{6n}-\frac{c\sqrt{n} }{2}=\frac{\sqrt{n}}{2}\left(\frac{\varSigma_{2}(\lambda)}{n^{3/2}}-c\right)=\frac{X_{2}(\lambda)}{2}.$$
\item[$k=2$.] $\scal{U_{2}(s-c)}{\Delta_{\lambda,c}(s)}=\scal{(s-c)^{2}-1}{\Delta_{\lambda,c}(s)}=\scal{(s-c)^{2}}{\Delta_{\lambda,c}(s)}$ is the second translated moment of the deviation, so it can be computed thanks to Lemma \ref{transmoment}. 
\begin{align*}
\frac{\sqrt{n}}{2}\scal{U_{2}(s-c)}{\Delta_{\lambda,c}(s)}&=\frac{\sqrt{n}}{12}\left(\frac{\tilp_{4}(\lambda)}{n^{2}}-4c\frac{\tilp_{3}(\lambda)}{n^{3/2}}+6c^{2}\frac{\tilp_{2}(\lambda)}{n}-6-4c^{2}\right)\\
&=\frac{\sqrt{n}}{12}\left(\frac{4\varSigma_{3}(\lambda)}{n^{2}}+\frac{6\varSigma_{1,1}(\lambda)}{n}+\frac{2}{n}-\frac{12c\varSigma_{2}(\lambda)}{n^{3/2}}+12c^{2}-6-4c^{2}\right)\\
&=\frac{\sqrt{n}}{3}\left(\frac{\varSigma_{3}(\lambda)}{n^{2}}-c^{2}\right)-c\sqrt{n}\left(\frac{\varSigma_{2}(\lambda)}{n^{3/2}}-c\right)-\frac{1}{3n^{1/2}}\\
&=\frac{X_{3}(\lambda)}{3}-cX_{2}(\lambda)-\frac{1}{3\sqrt{n}}
\end{align*}
Here, we've used the relation\footnote{Such a relation can be obtained if one considers the symbols $\varSigma_{\mu}$ as elements of the algebra of partial permutations, see \cite{IK99}. As a consequence, any product of terms $\varSigma_{\mu}$ is a sum over certain partial matchings $M$ of symbols $\varSigma_{\rho(M)}$, see \cite[\S3.3]{FM10} for a precise statement.} $\varSigma_{1,1}=(\varSigma_{1})^{2}-\varSigma_{1}$, and in general we shall use extensively the factorization of symbols $\varSigma_{\mu}$ in top homogeneous component.\vspace{2mm}
\item[$k=3$.] $\scal{U_{3}(s-c)}{\Delta_{\lambda,c}(s)}=\scal{(s-c)^{3}-2(s-c)}{\Delta_{\lambda,c}(s)}$ is the difference of the third translated moment, minus two times the first translated moment; this latter term will bring a contribution of $-X_{2}$. Then, by using again Lemma \ref{transmoment}, one sees that $\frac{\sqrt{n}}{2}\scal{(s-c)^{3}}{\Delta_{\lambda,c}(s)}$ is equal to
\begin{align*}
&\frac{\sqrt{n}}{20}\left(\frac{\tilp_{5}(\lambda)}{n^{5/2}}-5c\frac{\tilp_{4}(\lambda)}{n^{2}}+10c^{2}\frac{\tilp_{3}(\lambda)}{n^{3/2}}-10c^{3}\frac{\tilp_{2}(\lambda)}{n}+10c+5c^{3}\right)\\
&=\frac{\sqrt{n}}{20}\left(\frac{5\varSigma_{4}(\lambda)}{n^{5/2}}+\frac{20\varSigma_{2}(\lambda)}{n^{3/2}}-\frac{25\varSigma_{2}(\lambda)}{n^{5/2}}-\frac{20c\varSigma_{3}(\lambda)}{n^{2}}+\frac{20c}{n}+\frac{30c^{2}\varSigma_{2}(\lambda)}{n^{3/2}}-20c-15c^{3}\right)\\
&=\frac{\sqrt{n}}{4}\left(\frac{\varSigma_{4}(\lambda)}{n^{5/2}}-c^{3}\right)-c\sqrt{n}\left(\frac{\varSigma_{3}(\lambda)}{n^{2}}-c^{2}\right)+\left(\frac{3c^{2}\sqrt{n}}{2}+\sqrt{n}\right)\left(\frac{\varSigma_{2}(\lambda)}{n^{3/2}}-c\right)+\frac{c}{\sqrt{n}}-\frac{5\varSigma_{2}(\lambda)}{4n^{2}}\\
&=\frac{X_{4}(\lambda)}{4}-cX_{3}(\lambda)+\frac{3c^{2}X_{2}(\lambda)}{2}+X_{2}(\lambda)+\frac{c}{\sqrt{n}}-\frac{5\varSigma_{2}(\lambda)}{4n^{2}}.
\end{align*}
Again, we have used the relation $\varSigma_{2,1}=\varSigma_{2}\,\varSigma_{1}-2\varSigma_{2}$. By adding the contribution $-X_{2}(\lambda)$, one recovers the formula of Proposition \ref{almosttheorem}, because $\varSigma_{2}(\lambda)=O(n^{3/2})$ under a Schur-Weyl measure, and as a consequence $-5\varSigma_{2}(\lambda)/4n^{2}=O(n^{-1/2}) \to 0$. In the general case, we shall use the weight filtration on observables to neglict some terms: indeed, for any observable $f$, the order of magnitude is a $O(n^{\wt(f)/2})$, so one can for instance neglict terms of the kind $f/n^{k/2}$ with $k>\wt(f)$. Notice that there are some other simplifications in the calculation of $\scal{(s-c)^{3}}{\Delta_{\lambda,c}(s)}$; essentially, they are due to Lemma \ref{linearfunctional} that ensures that the mean of a scaled fluctuation is asymptotically equal to zero.\vspace{2mm}
\item[$k=4$.] Finally, another kind of trick has to be used in order to compute $\scal{U_{4}(s-c)}{\Delta_{\lambda,c}(s)}$, and this is due to the appearance of a term $\varSigma_{2,2}$ in the expansion of $\tilp_{6}$ in the basis of central characters. Because of Lemma \ref{transmoment},
$\frac{\sqrt{n}}{2}\scal{(s-c)^{4}}{\Delta_{\lambda,c}(s)}$ is equal to
\begin{align*}
&\frac{\sqrt{n}}{30}\left(\frac{\tilp_{6}(\lambda)}{n^{3}} - 6c\frac{\tilp_{5}(\lambda)}{n^{5/2}}+15c^{2}\frac{\tilp_{4}(\lambda)}{n^{2}}-20c^{3}\frac{\tilp_{3}(\lambda)}{n^{3/2}}+15c^{4}\frac{\tilp_{2}(\lambda)}{n}-20-15c^{2}-6c^{4}\right)\\
&=\frac{\sqrt{n}}{30}\left(\frac{6\varSigma_{5}(\lambda)}{n^{3}}+\frac{30\varSigma_{3}(\lambda)}{n^{2}}-\frac{90\varSigma_{3}(\lambda)}{n^{3}}+\frac{15(\varSigma_{2}(\lambda))^{2}}{n^{3}}-\frac{60}{n}+\frac{42}{n^{2}}\right.\\
&\left.-\frac{30c\varSigma_{4}(\lambda)}{n^{5/2}}-\frac{120c\varSigma_{2}(\lambda)}{n^{3/2}}+\frac{150c\varSigma_{2}(\lambda)}{n^{5/2}} +\frac{60c^{2}\varSigma_{3}(\lambda)}{n^{2}}-\frac{60c^{2}}{n}-\frac{60c^{3}\varSigma_{2}(\lambda)}{n^{3/2}}+75c^{2}+24c^{4}\right)
\end{align*}
In order to simplify the calculations, we have used the relations $\varSigma_{1,1,1}=(\varSigma_{1})^{3}-3(\varSigma_{1})^2+2\varSigma_{1}$, $\varSigma_{3,1}=\varSigma_{3}\,\varSigma_{1}-3\varSigma_{3}$ and $\varSigma_{2,2}=(\varSigma_{2})^{2}-4\varSigma_{3}-2(\varSigma_{1})^{2}+2\varSigma_{1}$. Now, the new trick consists in writing:
$$\frac{(\varSigma_{2}(\lambda))^{2}}{n^{3}}=\left(\frac{\varSigma_{2}(\lambda)}{n^{3/2}}\right)^{2}=\left(\frac{\varSigma_{2}(\lambda)}{n^{3/2}}-c\right)^{2}+2c\left(\frac{\varSigma_{2}(\lambda)}{n^{3/2}}-c\right)+c^{2}=\frac{(X_{2}(\lambda))^{2}}{n}+2c\frac{X_{2}(\lambda)}{\sqrt{n}}+c^{2}$$
Multiplying by the factor $\sqrt{n}$, one sees that the term $(X_{2}(\lambda))^{2}/\sqrt{n}$ will disappear in the limit $n \to \infty$, because $X_{2}(\lambda)=O(1)$. So, the previous expression may be written as:
\begin{align*}&\frac{X_{5}(\lambda)}{5}-cX_{4}(\lambda)+\left(2c^{2}+1\right)X_{3}(\lambda)-(3c+2c^{3})X_{2}(\lambda)\\
&+\frac{(X_{2}(\lambda))^{2}}{2\sqrt{n}}+\frac{5c\varSigma_{2}(\lambda)}{n^{2}}-\frac{3\varSigma_{3}(\lambda)}{n^{5/2}}-\frac{2+2c^{2}}{\sqrt{n}}+\frac{7}{5n^{3/2}}
\end{align*}
If one adds $\frac{\sqrt{n}}{2}\scal{-3(s-c)^{2}+1}{\Delta_{\lambda,c}(s)}$ in order to obtain $\frac{\sqrt{n}}{2}\scal{U_{4}(s-c)}{\Delta_{\lambda,c}(s)}$, one finally gets
\begin{align*}&\frac{X_{5}(\lambda)}{5}-cX_{4}(\lambda)+2c^{2}X_{3}(\lambda)-2c^{3}X_{2}(\lambda)\\
&+\frac{(X_{2}(\lambda))^{2}}{2\sqrt{n}}+\frac{5c\varSigma_{2}(\lambda)}{n^{2}}-\frac{3\varSigma_{3}(\lambda)}{n^{5/2}}-\frac{1+2c^{2}}{\sqrt{n}}+\frac{7}{5n^{3/2}}
\end{align*}
and every term on the second line can be neglicted when $n$ goes to infinity; on the other hand, the first line is precisely the formula of Theorem \ref{almosttheorem}. 
\end{enumerate}
Now, we certainly have enough numerical evidences in order to believe in Theorem \ref{almosttheorem}, and on the other hand, we also have a precise idea of the technicities needed for a general proof.
\end{examples}\bigskip

\begin{lemma}\label{idenpart}
For any positive integers $s\geq l$,
$$\sum_{\substack{|\mu|=s\\ \ell(\mu)=l}}\frac{1}{\prod_{i \geq 1}m_{i}(\mu)!}=\frac{1}{l!}\,\binom{s-1}{l-1}.$$
Similarly, given (random) variables $Y_{p\geq 1}$, one has
$$\sum_{\substack{|\mu|=s\\ \ell(\mu)=l}}\frac{1}{\prod_{i \geq 1}m_{i}(\mu)!}\left(\sum_{p \in \mu}Y_{p}\right)=\frac{1}{l-1!}\,\sum_{u=0}^{s-l} \binom{l-2+u}{u}\,Y_{s-l+1-u}.$$
\end{lemma}
\begin{proof}
The first formula is equivalent to
$$\sum_{\substack{|\mu|=s\\ \ell(\mu)=l}}\binom{l}{m_{1}(\mu),\ldots,m_{s}(\mu)}=\binom{s-1}{l-1},$$
and both sides of this equation correspond to the number of $l$-uplets of positive integers $(s_{1},\ldots,s_{l})$ such that $s_{1}+\cdots+s_{l}=s$, so this is indeed true. As for the second formula, one sees that:
\begin{align*}S_{2}&=\sum_{\substack{|\mu|=s\\ \ell(\mu)=l}}\frac{1}{\prod_{i \geq 1}m_{i}(\mu)!}\left(\sum_{p \in \mu}Y_{p}\right)=\sum_{\substack{|\mu|=s\\ \ell(\mu)=l}}\frac{1}{\prod_{i \geq 1}m_{i}(\mu)!}\left(\sum_{i=1}^{s-l+1}m_{i}(\mu)\,Y_{i}\right)\\
&=\sum_{\substack{|\mu|=s\\ \ell(\mu)=l}}\sum_{i=1}^{s-l+1}\frac{1}{\left(\prod_{\substack{j \geq 1\\ j \neq i}}m_{j}(\mu)!\right)\,(\max(0,m_{i}(\mu)-1))!}\,Y_{i}
\end{align*}
because the biggest possible part of a partition of length $l$ and size $s$ is $s-l+1$. Then, by reverting the order of summation, one obtains:
$$S_{2}=\sum_{i=1}^{s-l+1}Y_{i}\left(\sum_{\substack{|\mu|=s-i\\\ell(\mu)=l-1}} \frac{1}{\prod_{i \geq 1}m_{i}(\mu)!}\right)=\frac{1}{l-1!}\sum_{i=1}^{s-l+1}\binom{s-i-1}{l-2}\,Y_{i},$$
and this is exactly the second formula up to the change of index $u=s-l+1-i$.
\end{proof}\bigskip

\begin{lemma}\label{nasty}
Fix a parameter $c>0$, and a positive integer $k$. The scaled observable $\tilp_{k}(\lambda)/n^{\frac{k-1}{2}}$ is equal to
$$\sqrt{n}\left(\sum_{l=1}^{\lfloor \frac{k}{2}\rfloor} \frac{k!}{(k-l)\,k-2l!\,l-1!\,l!}\,c^{k-2l }\right)+\sum_{l=1}^{\lfloor \frac{k}{2}\rfloor}\frac{k!}{k-l!\,l-1!} \left( \sum_{u=0}^{k-2l} \binom{l-2+u}{u}c^{u}\,X_{1+k-2l-u} \right),$$
plus a random variable $V_{k}(c)$ that under the Schur-Weyl measures $\mathrm{SW}_{n,1/2,c}$ converges in probability towards a constant $L_{k}(c)$ as $n$ goes to infinity.
\end{lemma}
\begin{proof}
Because of Proposition \ref{changeofbasis}, one can write
$$\tilp_{k}=\left(\sum_{|\mu|+\ell(\mu)=k}\frac{k^{\downarrow \ell(\mu)}}{\prod_{i \geq 1}m_{i}(\mu)!} \prod_{i \geq 1}(\varSigma_{i})^{m_{i}}\right)+D_{k-1},$$
where $D_{k-1}$ is an observable of weight $k-1$. Let us decompose $D_{k-1}$ in the basis of free cumulants: 
$$D_{k-1}=\sum_{\substack{|\mu| \leq k-1\\ m_{1}(\mu)=0}}d_{\mu,k}\,R_{\mu}$$
Then, one knows that $R_{\mu}/n^{|\mu|/2}$ converges in probability towards $c^{|\mu|-2\ell(\mu)}$, so $V_{k}(c)=D_{k-1}/n^{\frac{k-1}{2}}$ converges in probability to
$$L_{k}(c)=\sum_{\substack{|\mu|=k-1  \\ m_{1}(\mu)=0}} d_{\mu,k}\,c^{|\mu|-2\ell(\mu)}.$$
We won't need the precise value of $L_{k}(c)$, so let us focus on the remaining part $\tilp_{k}-D_{k-1}$. One uses the same trick as for the term $(\varSigma_{2})^{2}$ in the expansion of $\widetilde{p}_{6}/n^{5/2}$:
\begin{align*}\frac{\tilp_{k}-D_{k-1}}{n^{\frac{k-1}{2}}}&=\sqrt{n}\left(\sum_{|\mu|+\ell(\mu)=k}\frac{k^{\downarrow \ell(\mu)}}{\prod_{i \geq 1}m_{i}(\mu)!} \prod_{i \geq 1}\left(\frac{\varSigma_{i}}{n^{\frac{i+1}{2}}}\right)^{m_{i}(\mu)}\right)\\
&=\sqrt{n}\left(\sum_{|\mu|+\ell(\mu)=k}\frac{k^{\downarrow \ell(\mu)}}{\prod_{i \geq 1}m_{i}(\mu)!} \prod_{i \geq 1}\left(\frac{X_{i}}{\sqrt{n}}+c^{i-1}\right)^{m_{i}(\mu)}\right)
\end{align*}
With $\mu$ fixed, one expands the product $\prod_{i\geq 1}\left(\frac{X_{i}}{\sqrt{n}}+c^{i-1}\right)^{m_{i}(\mu)}$ ; the two leading terms are
$$A(\mu)+\frac{B(\mu)}{\sqrt{n}}=c^{|\mu|-\ell(\mu)}+\frac{1}{\sqrt{n}}\sum_{p \in \mu} X_{p} \,c^{|\mu|-\ell(\mu)-p+1},$$
and the remainder is a $O(n^{-1})$. So,
$$\frac{\tilp_{k}-D_{k-1}}{n^{\frac{k-1}{2}}}=\sqrt{n}\left(\sum_{|\mu|+\ell(\mu)=k}\frac{k^{\downarrow \ell(\mu)}\,A(\mu)}{\prod_{i \geq 1}m_{i}(\mu)!}\right) + \left(\sum_{|\mu|+\ell(\mu)=k}\frac{k^{\downarrow \ell(\mu)}\,B(\mu) }{\prod_{i \geq 1}m_{i}(\mu)!}\right)+ O(n^{-1/2}).$$
The first part is equal to
$$\sqrt{n}\left(\sum_{l=1}^{\lfloor \frac{k}{2}\rfloor}\sum_{\substack{\mu \in \Part_{k-l}\\ |\mu|=l}} \frac{k!}{k-l!}
\,\frac{c^{k-2l }}{\prod_{i \geq 1}m_{i}(\mu)!}\right)=\sqrt{n}\left(\sum_{l=1}^{\lfloor \frac{k}{2}\rfloor} \frac{k!}{(k-l)\,k-2l!\,l-1!\,l!}\,c^{k-2l }\right)$$
because of the first part of Lemma \ref{idenpart}. As for the second part, we set $Y_{p}=X_{p}/c^{p-1}$ and we apply the second part of Lemma \ref{idenpart}, so we obtain
$$\sum_{l=1}^{\lfloor \frac{k}{2}\rfloor} \sum_{\substack{\mu \in \Part_{k-l}\\ \ell(\mu)=l}}\sum_{p \in \mu}\frac{k!}{k-l!}\,\frac{c^{k-2l}\,Y_{p}}{{\prod_{i \geq 1}m_{i}(\mu)!}}=\sum_{l=1}^{\lfloor \frac{k}{2}\rfloor}\frac{k!}{k-l!\,l-1!} \left( \sum_{u=0}^{k-2l} \binom{l-2+u}{u}Y_{1+k-2l-u} \right) c^{k-2l}.$$
Since $Y_{1+k-2l-u}\,c^{k-2l}=X_{1+k-2l-u}\,c^{u}$,  the proof is achieved.
\end{proof}\bigskip

\begin{lemma}
Up to a random variable that converges in probability to zero under the Schur-Weyl measures $\mathrm{SW}_{n,1/2,c}$, the translated $k$-th moment of the deviation $\frac{\sqrt{n}}{2}\int_{\R}(s-c)^{k}\,\Delta_{\lambda,c}(s)\,ds$ is equal to
$$\frac{1}{k+1}\sum_{l=0}^{k-1}\left(\sum_{m=0}^{\lfloor \frac{l}{2} \rfloor} \binom{k+1}{m}\binom{k+1-2m}{l-2m}\,(-c)^{l-2m}\right)X_{k+1-l}(\lambda).$$
\end{lemma}
\begin{proof}
Because of Lemma \ref{transmoment} and Lemma \ref{nasty}, the translated $k$-th moment of the deviation is the combination of the following terms:
\begin{align*}
A&=\sqrt{n}\,\left(\sum_{2\leq l \leq k+2}\sum_{m=1}^{\lfloor \frac{l}{2}\rfloor } \frac{k!}{(l-m)\,k+2-l!\,l-2m!\,m-1!\,m!}\,(-1)^{k+2-l}\, c^{k+2-2m}\,\right)\\
B&=\sum_{2\leq l \leq k+2}\sum_{m=1}^{\lfloor \frac{l}{2}\rfloor }\sum_{u=0}^{l-2m} \frac{k!}{k+2-l!\,l-m!\,m-1!}\binom{m-2+u}{u}(-1)^{k+2-l}\,c^{k+2-l+u}\,X_{1+l-2m-u}\\
C&=\sum_{2\leq l \leq k+2} \frac{k!}{l!\,k+2-l!}\,(-c)^{k+2-l}\,V_{l}(c)\\
D&=-\sqrt{n}\,\left(\sum_{2\leq 2m \leq k+2}\frac{k!}{m!\,k+2-m!}\,(-c)^{k+2-2m}\right)
\end{align*}
In $A$, we revert the order of summation:
$$\frac{A}{\sqrt{n}}=\sum_{2\leq 2m \leq k+2}\frac{k!}{m!\,k+2-2m!\,m-1!}(-c)^{k+2-2m}\left\{\sum_{l=2m}^{k+2}(-1)^{l}\frac{k+2-2m!}{(l-m)\,k+2-l!\,l-2m!}\right\}$$
The sum between brackets can be written as:
\begin{align*}\sum_{u=0}^{k+2-2m}\binom{k+2-2m}{u}\frac{(-1)^{u}}{u+m}&=\int_{0}^{1}\left(\sum_{u=0}^{k+2-2m}\binom{k+2-2m}{u}(-x)^{u}\right)x^{m-1}\,dx\\
&=\int_{0}^{1}(1-x)^{k+2-2m}\,x^{m-1}\,dx=\frac{k+2-2m!\,m-1!}{k+2-m!}
\end{align*}
Hence, $A+D=0$. Then, $B$ is asymptotically a centered gaussian variable, and $C$ converges towards $\sum_{2\leq l \leq k+2} \frac{k!}{l!\,k+2-l!}\,(-c)^{k+2-l}\,L_{l}(c)$. But because of Proposition \ref{linearfunctional}, the translated $k$-th moment of the deviation converges in law towards a centered gaussian variable, so
$$\sum_{2\leq l \leq k+2} \frac{k!}{l!\,k+2-l!}\,(-c)^{k+2-l}\,L_{l}(c)=0$$
 (this is why we did not need the precise values of the $L_{l}(c)$'s). Finally, one can simplify $B$ as follows. The index $1+l-2m-u$ takes its values in the interval $\lle 1,k+1\rre$, so $B$ can be written as a linear combination $\sum_{v=1}^{k+1} f(v,c)\,X_{v}$. Since $X_{1}=0$, one can actually take $v$ in $\lle 2,k+1 \rre$, so let us make the change of index $1+l-2m-u=k+1-x$ with $x \in \lle 0,k-1\rre$. Then,
 \begin{align*}
 B&=\sum_{x=0}^{k-1}X_{k+1-x}\left(\sum_{\substack{k+2\geq l \geq 2m \geq 2\\l-2m\geq k-x}} \frac{(-1)^{k+2-l}\,c^{x-2m+2}\,k!\,l-m-k+x-2!}{k+2-l!\,l-m!\,m-1!\,m-2!\,l-2m-k+x!}\right)\\
 &=\sum_{x=0}^{k-1}X_{k+1-x}\left(\sum_{y=0}^{\lfloor\frac{x}{2}\rfloor} \sum_{z=0}^{x-2y} \frac{(-1)^{x-2y-z}\,c^{x-2y}\,k!\,y+z-1!}{y!\,y-1!\,z!\,k+1-x+y+z!\,x-2y-z!}\right)
 \end{align*}
 with $y=m-1$ and $z=l-2m-k+x$. Then we use another hypergeometric identity
$$\sum_{z=0}^{Z} \frac{\alpha+z!}{\beta+z!}\binom{Z}{z} (-1)^{z}= \frac{\alpha!\,\beta-\alpha+Z-1!}{\beta-\alpha-1!\,\beta+Z!}$$
that holds for any $\beta> \alpha$, and that can again be shown by using methods \emph{\`a la} Wilf-Zeilberger, see \cite{PWZ97}. It implies that:
\begin{align*}
B&=\sum_{x=0}^{k-1}X_{k+1-x}\left(\sum_{y=0}^{\lfloor \frac{x}{2} \rfloor}\frac{(-c)^{x-2y}\,k!}{y!\,y-1!\,x-2y!}\sum_{z=0}^{Z=x-2y} \frac{y-1+z!}{y+k+1-x+z!}\binom{Z}{z}(-1)^{z}\right)\\
&=\sum_{x=0}^{k-1}X_{k+1-x}\left(\sum_{y=0}^{\lfloor \frac{x}{2} \rfloor}\frac{(-c)^{x-2y}\,k!\,k+1-2y!}{y!\,x-2y!\,k+1-x!\,k+1-y!}\right)\\
&=\frac{1}{k+1}\sum_{x=0}^{k-1}X_{k+1-x}\left(\sum_{y=0}^{\lfloor \frac{x}{2} \rfloor} \binom{k+1}{y}\binom{k+1-2y}{x-2y}\,(-c)^{x-2y}\right)
\end{align*}
Notice that a slight abuse of notation has been made when considering expressions such as $y-1!$ with $y=0$; nethertheless, it can be checked that in the end our formal computation does give the exact result.
\end{proof}
\begin{proof}[Proof of Theorem \ref{almosttheorem}]
The explicit expansion of $U_{k}(X)$ is
$$U_{k}(X)=\sum_{m=0}^{\lfloor \frac{k}{2}\rfloor} (-1)^{m}\binom{k-m}{m}\,X^{k-2m}$$
see \cite[p. 33]{IO02}. Moreover, given two set of variables $a_{0},a_{1},\ldots$ and $b_{0},b_{1},\ldots$, the following set of relations are equivalent
$$\left\{b_{k}=\sum_{m=0}^{\lfloor\frac{k}{2} \rfloor}\binom{k}{m} \,a_{k-2m}\right\}_{\!k \geq 0}\quad\iff\quad\left\{a_{k}=\sum_{m=0}^{\lfloor\frac{k}{2} \rfloor}(-1)^{m}\,\frac{k}{k-m}\binom{k-m}{m} \,b_{k-2m}\right\}_{\!k\geq 0}$$
see \cite[p. 36]{IO02}. We denote by $\frac{B_{k+1}}{k+1}$ the random variable computed in the previous lemma. Then, the linear functional of the deviation associated to the translated Chebyshev polynomial $U_{k}(s-c)$ is:
$$\frac{\sqrt{n}}{2}\scal{U_{k}(s-c)}{\Delta_{\lambda,c}(s)}=\sum_{m=0}^{\lfloor \frac{k}{2} \rfloor} (-1)^{m} \binom{k-m}{m}\,\frac{B_{k+1-2m}}{k+1-2m}=\sum_{m=0}^{\lfloor \frac{k}{2} \rfloor} (-1)^{m} \binom{k+1-m}{m}\,\frac{B_{k+1-2m}}{k+1-m}$$
Since $B_{0}=B_{1}=0$, the index $m$ can in fact be taken up to $\lfloor \frac{k+1}{2}\rfloor$, so the expression is exactly $\frac{A_{k+1}}{k+1}$, where the $A$'s and the $B$'s are related as mentioned before. Consequently, it is sufficient to identify the random variables $A_{k}$ such that:
$$\forall k,\,\,\,\sum_{l=0}^{k-1}\left(\sum_{m=0}^{\lfloor \frac{l}{2} \rfloor} \binom{k+1}{m}\binom{k+1-2m}{l-2m}\,(-c)^{l-2m}\right)X_{k+1-l}=\sum_{m=0}^{\lfloor \frac{k+1}{2}\rfloor} \binom{k+1}{m}\,A_{k+1-2m}$$
In the left-hand side, we revert the order of summation and thus obtain:
\begin{align*}B_{k+1}&=\sum_{m=0}^{\lfloor \frac{k-1}{2}\rfloor}\binom{k+1}{m}\left(\sum_{l=2m}^{k-1}\binom{k+1-2m}{l-2m}\,(-c)^{l-2m}\,X_{k+1-l}\right)\\ 
&=\sum_{m=0}^{\lfloor \frac{k-1}{2}\rfloor}\binom{k+1}{m}\left(\sum_{l=0}^{k-1-2m}\binom{k+1-2m}{l}\,(-c)^{l}\,X_{k+1-2m-l}\right)\\
&=\sum_{m=0}^{\lfloor \frac{k+1}{2}\rfloor}\binom{k+1}{m}\left(\sum_{l=0}^{k-1-2m}\binom{k+1-2m}{l}\,(-c)^{l}\,X_{k+1-2m-l}\right)\quad\text{because }X_{0}=X_{1}=0.
\end{align*}
As a consequence,
$$A_{k+1}=\sum_{l=0}^{k-1}\binom{k+1}{l}\,(-c)^{l}\,X_{k+1-l},$$
and since $\frac{A_{k+1}}{k+1}=\frac{\sqrt{n}}{2}\scal{U_{k}(s-c)}{\Delta_{\lambda,c}(s)}$,  the proof of Theorem \ref{almosttheorem} is completed.
\end{proof}
\bigskip
\bigskip

\section{An extension of Kerov's central limit theorem}\label{ksw}
Let us denote by $(X_{k,\infty})_{k \geq 2}$ the limit in law of the random process $(X_{k})_{k \geq 2}$ under the Schur-Weyl measures $\mathrm{SW}_{n,1/2,c}$; the covariance matrix of this centered gaussian process has been computed in \S\ref{sniady}. One sees easily that in the gaussian space generated by the process $(X_{k,\infty})_{k \geq 2}$, the orthogonal basis deduced from $(X_{k,\infty})_{k \geq 2}$ by performing the Gram-Schmidt orthogonalisation procedure is:
$$ Y_{k,\infty}=\sum_{l=0}^{k-2} \binom{k}{l}(-c)^{l}\,X_{k-l,\infty}$$
As a consequence, the joint limit in law of the random variables $(\frac{\sqrt{n}}{2}\scal{U_{k}(s-c)}{\Delta_{\lambda,c}(s)})_{k\geq 1}$ is simply $(\frac{Y_{k+1,\infty}}{k+1})_{k\geq 1}$. So, Theorem \ref{almosttheorem} can be restated in the following way:
\begin{theorem}[Central limit theorem for Schur-Weyl measures]\label{global}
Under the measures $\mathrm{SW}_{n,1/2,c}$, the random process of scaled deviations
$$\left(\frac{\sqrt{n}}{2}\scal{U_{k}(s-c)}{\Delta_{\lambda,c}(s)}\right)_{k\geq 1}$$
converges in finite dimensional laws towards a gaussian process $(\xi_{k+1})_{k \geq 1}$ whose coordinates $\xi_{k+1}$ are independant normal variables of mean $0$ and variance $1/(k+1)$. Consequently, in the interval $[c-2,c+2]$, the scaled deviation $\frac{\sqrt{n}}{2}\Delta_{\lambda,c}(s)$ converges in distribution towards the generalized gaussian process
$$\Delta_{c}(s)=\Delta_{c}(c+2\cos \theta)=\frac{1}{\pi}\sum_{k=2}^{\infty} \frac{\xi_{k}}{k}\,\sin(k\theta).$$
\end{theorem}
\noindent Hence, Kerov's central limit theorem as enounced in \cite{IO02} for Plancherel measures also holds for Schur-Weyl measures, with exactly the same limit up to a translation on the $x$-axis. In other words, Kerov's gaussian process has a universality property. It would be very interesting to extend this universality property further; as for now, it looks more as a miracle of computations than as a profound result about random partitions. 

\figcap{\includegraphics[scale=0.7]{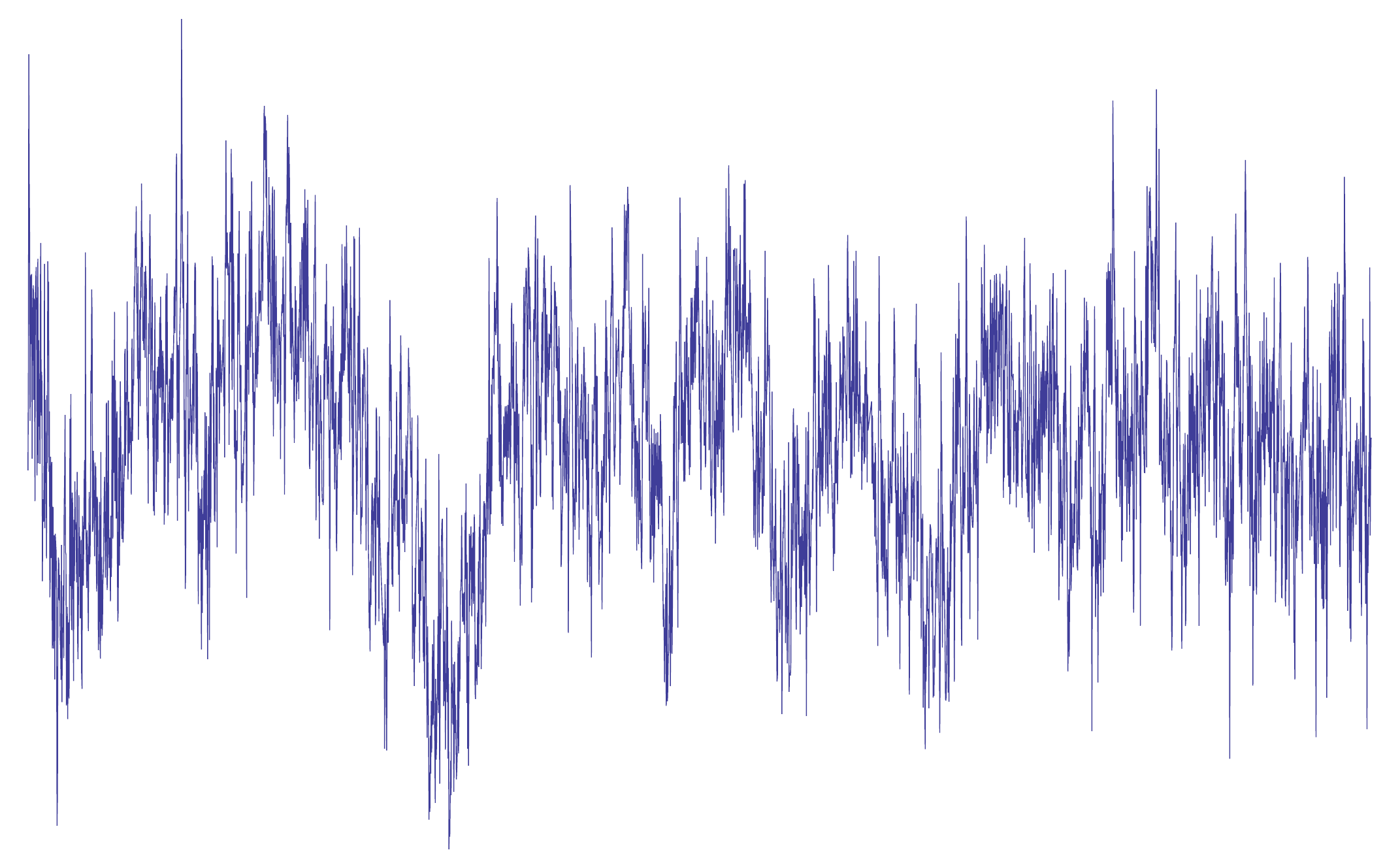}}{Partial sum of the 5000 first terms of the random serie that defines the limiting gaussian process associated to the deviations $\Delta_{\lambda,c}(s)$.\label{gaussianprocess}}

Notice that the random serie $\Delta_{c}(s)$ does not converge ponctually --- this is the same phenomenon as for the well-known free gaussian field, see \cite{She07} and also our figure \ref{gaussianprocess}. However, it makes sense as a distribution, meaning that for any function $f \in \mathscr{C}^{\infty}([c-2,c+2])$, $$\frac{\sqrt{n}}{2}\scal{f(s)}{\Delta_{\lambda,c}(s)} \to \frac{1}{\pi}\sum_{k=2}^{\infty} \frac{\xi_{k}}{k} \left(\int_{c-2}^{c+2} f(s)\,\sin(k\theta(s))\,ds\right),$$
where the right-hand side converges with probability $1$.

\begin{proof}
The first part has just been proved (set $\xi_{k+1}=\frac{Y_{k+1,\infty}}{k+1}$), and for the second part, it is exactly the same discussion as in \cite[\S9]{IO02}.
\end{proof}
\bigskip

\bibliographystyle{alpha}
\bibliography{kerovschurweyl}

\end{document}